\documentclass[12pt]{article}

\pdfoutput=1

\usepackage{amssymb}
\usepackage{amsthm}
\usepackage{lineno}
\usepackage[hidelinks]{hyperref}                         
\usepackage{xcolor} 
\usepackage{amsmath}
\usepackage{pgf}
\usepackage{tikz} 
\usetikzlibrary{arrows, shapes, matrix, decorations.markings}

\usepackage[utf8]{inputenc}
\usepackage[T1]{fontenc}

\usepackage{aliascnt}

\theoremstyle{plain}
\newtheorem{theorem}{Theorem}[section] 
\newaliascnt{proposition}{theorem}
\newtheorem{proposition}[proposition]{Proposition}
\aliascntresetthe{proposition}
\newaliascnt{lemma}{theorem}
\newtheorem{lemma}[lemma]{Lemma}
\aliascntresetthe{lemma}
\newaliascnt{corollary}{theorem}
\newtheorem{corollary}[corollary]{Corollary}
\aliascntresetthe{corollary}
\newaliascnt{conjecture}{theorem}
\newtheorem{conjecture}[conjecture]{Conjecture}
\aliascntresetthe{conjecture}
\newaliascnt{prop}{theorem}
\newtheorem{prop}[prop]{Proposition}
\aliascntresetthe{prop}
\newtheorem{thm}[theorem]{Theorem}
\newaliascnt{lem}{theorem}
\newtheorem{lem}[lem]{Lemma}
\aliascntresetthe{lem}
\newaliascnt{cor}{theorem}
\newtheorem{cor}[cor]{Corollary}
\aliascntresetthe{cor}
\newaliascnt{conj}{theorem}

\aliascntresetthe{conj}
\newaliascnt{remark}{theorem}

\aliascntresetthe{remark}
\newaliascnt{note}{theorem}

\aliascntresetthe{note}

\theoremstyle{definition}
\newaliascnt{statement}{theorem}
\newtheorem{statement}[statement]{Statement}
\aliascntresetthe{statement}
\newaliascnt{definition}{theorem}
\newtheorem{definition}[definition]{Definition}
\aliascntresetthe{definition}
\newaliascnt{define}{theorem}
\newtheorem{define}[define]{Definition}
\aliascntresetthe{define}
\newaliascnt{defn}{theorem}

\aliascntresetthe{defn}
\newaliascnt{question}{theorem}
\newtheorem{question}[question]{Question}
\aliascntresetthe{question}

\usepackage{cleveref}

\crefname{proposition}{proposition}{propositions}
\crefname{lemma}{lemma}{lemmas}
\crefname{corollary}{corollary}{corollaries}
\crefname{conjecture}{conjecture}{conjectures}
\crefname{prop}{proposition}{propositions}
\crefname{lem}{lemma}{lemmas}
\crefname{cor}{corollary}{corollaries}
\crefname{conj}{conjecture}{conjectures}
\crefname{remark}{remark}{remarks}
\crefname{note}{note}{notes}
\crefname{statement}{statement}{statements}
\crefname{definition}{definition}{definitions}
\crefname{define}{definition}{definitions}
\crefname{defn}{definition}{definitions}
\crefname{question}{question}{questions}

\renewcommand{\hat}{\widehat}
\def\halts{\!\downarrow}

\def\la{\langle}
\def\ra{\rangle}
\def\ss{\scriptsize}
\def\M{\mathcal{M}}
\def\O{\mathcal{O}}
\newcommand{\st}{:}
\newcommand{\inM}[1]{#1$^{\M}$}
\newcommand{\inside}[1]{{#1}^\circ}
\newcommand{\upto}{\!\upharpoonright}  

\DeclareMathOperator{\dom}{\mathrm{dom}}

\newcommand{\N}{\mathbb{N}}
\newcommand{\HYP}{\mathrm{HYP}}

\DeclareMathOperator{\ACA}{\mathsf{ACA}_0}
\DeclareMathOperator{\ATR}{\mathsf{ATR}_0}
\DeclareMathOperator{\TLPP}{\mathsf{TLPP}_0}
\DeclareMathOperator{\PCA}{\Pi^1_1\mbox{-}\mathsf{CA}_0}
\DeclareMathOperator{\PCAp}{\Pi^1_1\mbox{-}\mathsf{CA}_0^+}
\DeclareMathOperator{\PTI}{\Pi^1_1\mbox{-}\mathsf{TI}_0}
\def\PtCA{\Pi^1_2\mbox{-}\mathsf{CA}_0}

\def\PTR{\Pi^1_1\mbox{-}\mathsf{TR}_0}
\DeclareMathOperator{\RCA}{\mathsf{RCA}_0}
\DeclareMathOperator{\SAC}{\Sigma^1_1\mbox{-}\mathsf{AC}_0}
\DeclareMathOperator{\SDC}{\Sigma^1_1\mbox{-}\mathsf{DC}_0}

\DeclareMathOperator{\WKL}{\mathsf{WKL}_0}

\DeclareMathOperator{\CKDT}{\mathsf{CKDT}}
\def\ST{\mathsf{PM}}
\def\MIM{\mathsf{MIM}}
\def\Maximality{\mathsf{MM}}
\newcommand{\FST}{\text{Finite }\ST}
\newcommand{\cFB}{\text{Bounded }\ST}
\newcommand{\FB}{\text{Locally Finite }\ST}
\newcommand{\FPST}{\text{Finite Path }\ST}
\newcommand{\CPM}{\text{Collection of }\mathsf{PM}}

\begin{document}

\title{The computational strength of matchings\\ in countable graphs} 
\date{November 12, 2019}
\author{Stephen Flood, Matthew Jura, Oscar Levin, Tyler Markkanen}

\maketitle

\begin{abstract}
In a 1977 paper, Steffens identified an elegant criterion for determining when a countable graph has a perfect matching.  
In this paper, we will investigate the proof-theoretic strength of this result and related theorems.
We show that a number of natural variants of these theorems are equivalent, or closely related, to the ``big five'' subsystems of reverse mathematics.

The results of this paper explore the relationship between graph theory and logic by showing the way in which specific changes to a single graph-theoretic principle impact the corresponding proof-theoretical strength.
Taken together, the results and questions of this paper suggest that the existence of matchings in countable graphs provides a rich context for understanding reverse mathematics more broadly.

\end{abstract}

\section{Introduction}

The search for necessary and sufficient conditions for a graph to have a perfect matching has a long and distinguished history, and has resulted in a range of theorems both for bipartite graphs and for graphs in general.  
For theorems concerning matchings in bipartite graphs, much research has been done to classify their computational and proof-theoretical strength \cite{aharoni_magidor,shafer,simpsonkdt,towsner}.
 
In this paper, we will turn our attention to classifying the strength of theorems about the existence of matchings in \emph{general} graphs.  In particular, we will study a number of related theorems due to Steffens~\cite{steffens}, who gave a particularly elegant criterion for the existence of perfect matchings in countable graphs.

\subsection{The existence of perfect matchings}
Given a graph $G = (V, E)$ with vertices $V$ and edges $E$, a \emph{matching} is a subset $M \subseteq E$ so that no vertex of $G$ is incident to more than one edge of $M$.  A matching is \emph{perfect} if every vertex of $G$ is incident to an edge of $M$.   

Now suppose that $M$ is some matching, and that $v$ is a vertex not covered by $M$.  It is natural to ask whether $M$ can be modified to cover $v$, by starting at $v$ and repeatedly adding and removing edges.  To make this precise, consider a matching $M$ and a path $P=(v_i)_{i < k \le \omega}$, where a \emph{path} $P$ is an injective sequence of vertices such that $\{v_{i},v_{i + 1}\} \in E$ for each $i+1<k$.

A path is called \emph{$M$-alternating} if the edges $\{v_{i},v_{i + 1}\}$ alternately lie in $M$ and $E \setminus M$, and an $M$-alternating path $P$ is called \emph{$M$-augmenting} if it starts at a vertex $s \in V(G \setminus M)$ and either (1) $P$ is infinite, or (2) $P$ terminates in a vertex $v \in V(G \setminus M)$.  
Notice that if $P$ is an $M$-augmenting path, then swapping the membership of those edges in the path results in a strictly larger matching.  In other words, the existence of an $M$-augmenting path allows us to extend the matching $M$ to cover an additional one or two vertices.  We often identify $P$ with its constituent edges, so this process of alternating the edges of $M$ which are a part of $P$ is the same as taking the symmetric difference $M \mathbin{\Delta} P$ of $M$ and the edges of $P$.

It is natural to ask whether there is a connection between the existence of perfect matchings and the existence of augmenting paths for imperfect matchings.  For countable graphs, this question was answered  by Steffens~\cite{steffens} using the following terminology.\footnote{For uncountable graphs $G$, there is a different condition classifying the $G$ which have a perfect matching, but that condition is more complicated.  See Aharoni's results of \cite{aharonimatchings}.}

\begin{definition}
A countable graph $G$ is said to satisfy \emph{condition~(A)} if for every matching $M$ and for every vertex $s \in V(G \setminus M)$ there exists an $M$-augmenting path which starts at $s$.
\end{definition}

\begin{theorem}[Steffens~\cite{steffens}] \label{steffensThm}
  A countable graph $G$ has a perfect matching if and only if $G$ satisfies condition~(A). 
\end{theorem}

The implication in one direction is straightforward.

\begin{prop} \label{prop-perfect-to-A}
If a countable graph $G$ has a perfect matching, then $G$ satisfies condition~(A).  Furthermore, this holds over $\RCA$.
\end{prop}

\begin{proof}
Fix a graph $G$ and a perfect matching $N$.  
For any imperfect matching $M$ and any $s \notin V(M)$, comparing $M$ with $N$ allows us to define an $M$-augmenting path $(v_i)_{i < k \leq \omega}$, as required by the statement of condition~(A).  
To see why, let $v_0 = s$ and $v_1$ be the neighbor of $s$ such that $\{s, v_1\} \in N$.  
If $v_1 \notin V(M)$, then $P=(s,v_1)$ is the desired $M$-augmenting path. Otherwise, if $v_1\in V(M)$, there is a $v_2$ such that $\{v_1, v_2\} \in M$.  
But then, because $N$ matched $v_1$ to $s$ and because $v_2\in V(N)$, there must be a distinct $v_3$ such that $\{v_2, v_3\} \in N$. 
Continuing like this, either we end with $v_k$ outside of $V(M)$ or we never end.  In either case, the path is $M$-augmenting.
Furthermore, this path is clearly $\Delta^0_1$ in $M$ and $N$, so the path exists by $\RCA$.
\end{proof}

It follows that the strength of \Cref{steffensThm} is contained in the implication ``If condition~(A) holds of a graph $G$, then there is a perfect matching of $G$.''  We will call this direction of the biconditional the \textit{Perfect Matching Theorem}, or $\ST$.  

Given a matching $M$, the corresponding vertex set $V(M)$ is called the \emph{support} of $M$.  Steffens points out that his proof of \Cref{steffensThm} proves an apparently stronger result: even graphs that do not satisfy condition~(A) must have a matching of maximal support.
 
\begin{theorem}[Steffens~\cite{steffens}] \label{maximalityThm}
  For each countable graph $G = (V, E)$ there is a maximal subset $V' \subseteq V$ which has a perfect matching.
\end{theorem}
 
We will call \Cref{maximalityThm} the \emph{Maximal Matching Theorem}, or $\Maximality$.  
The connection between the main theorems is straightforward, but illuminating.

\begin{proposition}\label{prop.MAXimpliesPM}
\Cref{maximalityThm} implies \Cref{steffensThm} over $\RCA$.
\end{proposition}

\begin{proof}
Consider any graph $G$.  By \Cref{maximalityThm}, $G$ has a matching $M$ of maximal support.  Suppose also that $G$   satisfies condition~(A).  Then if there is any $v \notin V(M)$, there is an $M$-augmenting path $P$ starting at that $v$.  But recall from above that because $P$ is an $M$-augmenting path, then $M \mathbin{\Delta} P$ is a matching of $G$ that covers more vertices than $M$.  This contradicts the maximality of $M$. 
Furthermore, it is clear that $M \mathbin{\Delta} P$ is $\Delta^0_1$ in $M$ and $P$, so the proof goes through in $\RCA$.
\end{proof}

The proofs of \Cref{maximalityThm,steffensThm} revolve around a special kind of matching, which will be discussed in more depth in \Cref{sect.matchings-in-general}.

\begin{definition}
A matching $M$ is \emph{independent} if the only $M$-augmenting paths are length-$1$ paths that begin and end at vertices outside of $M$.
\end{definition}

Although it is also true that every graph has an edge maximal matching, edge maximal matchings are trivial to construct, and so are not of interest to this paper.

\subsection{Classifying computational strength}
We will analyze several versions of $\ST$ and $\Maximality$ (\Cref{steffensThm,maximalityThm}) from the viewpoint of computability theory and reverse mathematics.
Indeed, this paper can be seen as a ``case study'' in how varying specific graph-theoretic features of a mathematical principle directly impacts the logical strength of that principle.  

In the next section, \Cref{figure.ladder} summarizes our results showing that versions of \Cref{steffensThm,maximalityThm}, obtained by restricting them to different classes of graphs, are either equivalent or closely related to the standard subsystems of second order arithmetic.

This work continues several important veins of research into the reverse mathematics of principles in infinite graph theory.  
In particular, the current paper is closely related to the reverse mathematics of the K{\"o}nig duality theorem for countable bipartite graphs, written $\CKDT$, which relates the existence of matchings to vertex covers of bipartite graphs. 
This theorem has been shown to be equivalent to $\ATR$.  A proof in $\PCA$ and a reversal to $\ATR$ were provided by Aharoni, Magidor, and Shore~\cite{aharoni_magidor}, and  Simpson~\cite{simpsonkdt} later showed that this theorem is provable in $\ATR$.
Continuing this line of work, Shafer \cite{shafer} studied Menger's theorem for countable webs, which is a sort of generalization of $\CKDT$.
Shafer, in \cite{shafer}, showed that Menger's theorem is provable in $\PCA$, and defined an extended version that is equivalent to $\PCA$.  Later, Towsner \cite{towsner} refined this result to show that Menger's theorem for countable bipartite graphs is actually provable in the system that he defines in \cite{towsner} and calls $\TLPP$ (which stands for ``transfinite leftmost path principle''), which lies strictly between $\ATR$ and $\PCA$.

The work in the current paper continues this trend of increasing complexity.  It is not difficult to see that the statement itself of $\ST$ is significantly more complex than either $\CKDT$ or Menger's theorem for countable webs.  In particular, those two theorems are defined by $\Pi^1_2$ sentences, while $\ST$ is equivalent to a $\Pi^1_3$ sentence.  
In other words, $\ST$ is particularly interesting since it has a similar flavor to these other principles, but with a significantly higher sentence complexity.  

In other ways, this paper builds on research such as that of
Hirst \cite{hirst}, \cite{hirst-hughes} who studied a number of of variants of Hall's theorem concerning matchings of bipartite graphs.
By considering different necessary and sufficient conditions for a graph to have a perfect matching, Hirst obtained principles equivalent to systems around the levels of $\ACA$ and $\WKL$. 
The current paper continues and extends this line of research,  supporting and expanding our understanding of the relationship between specific graph-theoretic features and reverse mathematics.

We assume that the reader is familiar with computability theory and reverse mathematics, including the ``big five'' subsystems of second order arithmetic: $\RCA$, $\ACA$, $\WKL$, $\ATR$, and $\PCA$.  Other subsystems will be defined when they are introduced.  For additional background on computability theory, see \cite{rogers} or \cite{soare}.  For additional background on reverse mathematics, see \cite{simpson}.

\section{Summary of results}

We will study the existence of three main types of matchings in countable graphs: perfect matchings, matchings of maximal support, and maximal independent matchings.  More precisely, we will use the following definitions, formalized in second order arithmetic.

\newcommand{\MarkLt}{4pt}
\newcommand{\MarkSep}{3pt}
\tikzset{
  TwoMarks/.style={
    postaction={decorate,
      decoration={
        markings,
        mark=at position #1 with
          {
              \begin{scope}[xslant=0.2]
              \draw[line width=\MarkSep,white,-] (0pt,-\MarkLt) -- (0pt,\MarkLt) ;
              \draw[-] (-0.5*\MarkSep,-\MarkLt) -- (-0.5*\MarkSep,\MarkLt) ;
              \draw[-] (0.5*\MarkSep,-\MarkLt) -- (0.5*\MarkSep,\MarkLt) ;
              \end{scope}
          }
       }
    }
  },
  TwoMarks/.default={0.5}
}  
\begin{figure}[th]
\begin{center}
  \begin{tikzpicture}[->, > = stealth, scale = 1.25]
    \pgfsetlinewidth{0.8 pt}
    \node at (0, 8) (P2CAplus) {$\PtCA^+$};
    \node at (5, 6.5) (MAX)  {$\Maximality$};
    \node at (7.5, 6.5) (SteffensPlusL3)  {$\ST+\MIM$};
    \node at (0,7) (P2CA) {$\PtCA$}; 
    \node at (2,5.5) (MAXind) {$\MIM$};
    \node at (0,5) (PCA)  {$\PCA$};
    \node at (5,5.5) (seqStef) {Sequential $\ST$};
    \node at (0,4) (ATR)  {$\ATR$};
    \node at (5,4.5) (Stef)   {$\ST$};
    \node at (3,4.00) (FP-Stef)   {$\FPST$};
    \node[right] at (2.5,3) (CPM)   {$\CPM$};
    \coordinate (both) at (4,5.5);
    \node at (0,3) (S11AC){$\SAC$};
    \node at (0,2) (ACA)  {$\ACA$};
    \node[right] at (2.5,2) (FB)    {$\FB$; Locally Finite $\Maximality$};
    \node at (0,1) (WKL)  {$\WKL$}; 
    \node[right] at (2.5,1) (cFB)  {$\cFB$; Bounded $\Maximality$};
    \node at (0,0) (RCA)   {$\RCA$};
    \node[right] at (2.5,0) (Finite) {$\FST$; Finite $\Maximality$};
    \tikzstyle{every node}=[font=\tiny\itshape]
    \draw[TwoMarks] (MAX) -- (P2CA);
    \draw[double distance = 1 pt] (P2CAplus) -- (MAX) node [midway, below left] {\ss Cor.\ \ref{proofofMAX}};
    \draw[double distance = 1 pt] (P2CA) -- (MAXind) node[midway,right] {\ss Thm.\ \ref{lemma.SteffensLemma3.p12ca}};
    \draw[<->] (MAX) -- (SteffensPlusL3) node [midway, above] {\ss Thm.\ \ref{PMplusMIMimpliesMAX}};
    \draw (MAX) --  (MAXind)  node [midway, above left] {\ss Cor.\ \ref{cor.MAXimpliesMIM}};
    \draw (MAXind) -- (PCA) node [midway,above,xshift=-.2cm,yshift=.1cm] {\ss Thm.\ \ref{thm.MIM.implies.PCA}};
    \draw[line width=.8pt] (MAX) -- node [right]{\ss Prop.\ \ref{prop.MaxImpliesSequential}}(seqStef);
    \draw (seqStef) -- (Stef);
    \draw (seqStef) -- node[below right] {\ss Prop.\ \ref{seqToPca}} (PCA);
    \draw (Stef) -- (FP-Stef);
    \draw[dashed] (FP-Stef) -- (ATR) node[midway, above]{\tiny $+\Pi^1_1\mbox{-}\mathsf{TI}_0$} node[midway, below]{\ss Thm.\ \ref{reversetoATR}};
    \draw[double distance = 1 pt] (P2CAplus) -- (P2CA);
    \draw[double distance = 1 pt] (P2CA) -- (PCA);
    \draw[double distance = 1 pt] (PCA) -- (ATR);
    \draw[double distance = 1 pt] (ATR) -- (S11AC);
    \draw[double distance = 1 pt] (S11AC) -- (ACA);
    \draw[double distance = 1 pt] (ACA) -- (WKL);
    \draw[double distance = 1 pt] (WKL) -- (RCA);
    \draw [double distance = 1 pt] (Stef) -- node[below right] {\ss Thm.\ \ref{StefToS11AC} \& \ref{corollary.steffens-not-consequence-of-Sigma11AC}} (CPM);
    \draw [double distance = 1 pt] (PCA) -- node[above right] {\ss Cor.\ \ref{fpstinpi11ca}}  (FP-Stef);
    \draw[<->] (CPM) -- (S11AC) node [midway, below] {\ss Cor.\ \ref{disjointunion}};
    \draw[<->] (FB) -- node[below] {\ss Prop.\ \ref{acaIsFB} \& \ref{prop.MMforLocFin}} (ACA);
    \draw[<->] (cFB) -- node[below] {\ss Prop.\ \ref{wklIsCFB} \& \ref{prop.MMforLocFin}} (WKL);
    \draw[<->] (RCA) -- node[below] {\ss Prop.\ \ref{rcaToFinite}} (Finite);
  \end{tikzpicture}%
\end{center}%
\caption{Summary of results}
\label{figure.ladder}
\end{figure}
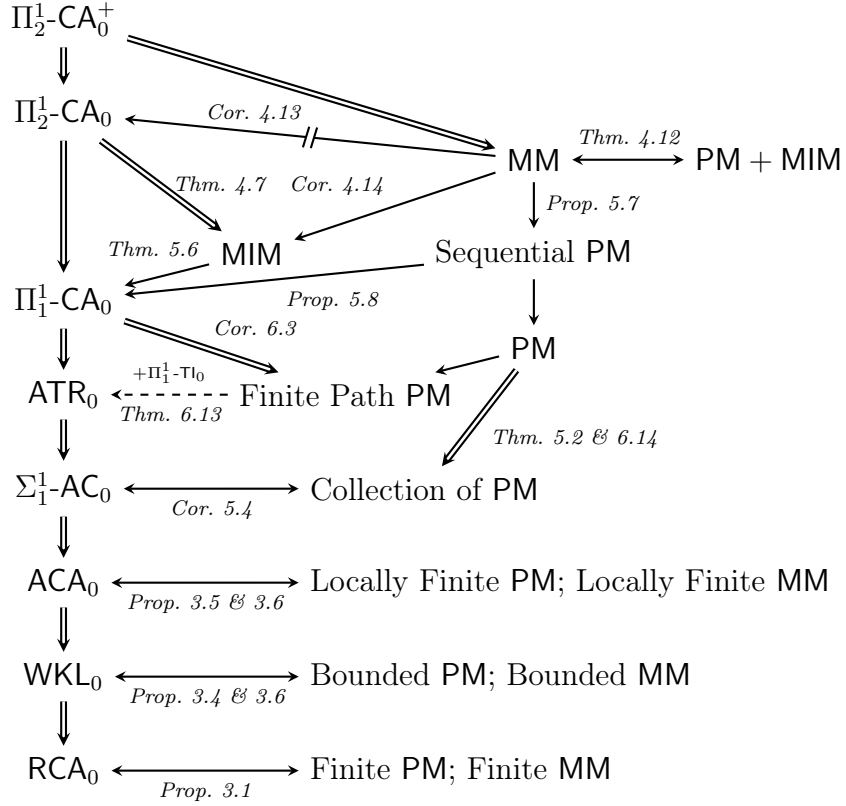

\begin{statement}\label{statements} 
The following statements can be formalized in $\RCA$.
  \begin{enumerate}
    \item The Perfect Matching Theorem, or $\ST$, denotes the statement,
    \emph{if $G$ is a graph satisfying condition~(A), then there is a perfect matching $M$ of $G$.}
    \par 
\noindent (We also study the natural restriction of $\ST$ to specific classes of graphs, including finite graphs, bounded graphs, and locally finite graphs.) 

    \item The Maximal Matching Theorem, or $\Maximality$, denotes the statement,
    \emph{for any graph $G$, there is a matching $M$ of $G$ so that $V(M)$ is not properly contained in the support of any other matching.}

    \item Maximal Independent Matching, or $\MIM$, 
    denotes the statement,
    \emph{for any graph $G$, there is an independent  matching $M$ of $G$ so that $V(M)$ is not properly contained in the support of any other independent matching.}

    \item $\CPM$ denotes the statement,
    \emph{if $X = \langle G_1, G_2, \dots \rangle $ codes a sequence of disjoint graphs, and if each $G_i$ has a perfect matching $M_i$, then there is a perfect  matching $M$ for the graph $G = \bigcup_i G_i$. }

    \item Sequential $\ST$ denotes the statement, \emph{if $X = \langle G_1, G_2, \dots \rangle $ codes a sequence of disjoint graphs, then $\ST$ holds of each $G_i$}.
\end{enumerate}
\end{statement}

The majority of the results of this paper are summarized in \Cref{figure.ladder}, which shows relationships between these principles and standard subsystems of second order arithmetic.  Solid arrows indicate a proof over $\RCA$, while the dashed arrow indicates a reversal over the specified base system.  When a strict implication is known to hold, we use a double arrow.  A non-redundant non-implication is indicated using a slash through an arrow.

During the course of the paper, we will study a number of versions of $\ST$ obtained by restricting the statement of $\ST$ to a specific class of graphs.  
In \Cref{sec.locallyfinite}, we study locally finite graphs. 
There, $\FB$ denotes the restriction of $\ST$ to graphs in which each vertex has finitely many neighbors. 
We show that $\FB$ is equivalent to $\ACA$ over $\RCA$. 
Similarly, 
$\cFB$ denotes the restriction of $\ST$ to only those locally finite graphs where there exists a function bounding the neighborhood relation.
  In other words, $\cFB$ is the restriction of $\ST$ to the reverse mathematical analogue of highly computable graphs. 
  We show that $\cFB$ is equivalent to $\WKL$ over $\RCA$.
These results illustrate and support the standard reverse mathematical intuitions concerning the combinatorial content of $\ACA$ and $\WKL$.  

The picture becomes more interesting in Sections \ref{sect.matchings-in-general}, \ref{sect.lower-bounds-general}, and \ref{sect.noinfpaths} when considering graphs that are not locally finite.  The principles concerning maximal matchings, $\Maximality$ and $\MIM$, are both above $\PCA$, while $\MIM$ is strictly below $\PtCA$, and $\Maximality$ is below $\PtCA^+$, but cannot imply $\PtCA$.  
The strength of $\ST$ falls somewhere between $\Maximality$ and $\FPST$, 
where $\FPST$ denotes the restriction of $\ST$ to graphs with no infinite paths (recall that a path cannot visit any vertex twice).  We show that $\FPST$ implies $\ATR$ over $\PTI$.

To see why this base system is so unusual, recall that the hypothesis, ``$G$ satisfies condition~(A),'' is $\Pi^1_2$ with parameter $G$, while the conclusion is $\Sigma^1_1$ (also with parameter $G$).  In other words, the hypothesis of $\ST$ is more complex than its conclusion.
This is in stark contrast with theorems such as K{\"o}nig's Lemma, whose hypothesis is a $\Pi^0_1$ sentence (the existence of arbitrarily long finite paths) and whose conclusion is a $\Sigma^{1}_1$ sentence (the existence of an infinite path).  When the conclusion is of higher complexity than the hypothesis, it is generally straightforward to verify the properties of any attempted coding.  In $\ST$, on the other hand, the hypothesis is more complex than the conclusion.  Thus, any attempt to code into $\ST$ requires proving a $\Pi^1_2$ property in order to draw a $\Sigma^1_1$ conclusion.
This discussion suggests one of several novel challenges introduced by the complexity of condition~(A), which will be a common theme throughout this paper.

\section{Matchings for locally finite graphs}
\label{sec.locallyfinite}

We begin by considering restrictions of $\ST$ to the simpler case of locally finite graphs. 

It is easy to see that $\ST$ for finite graphs is provable in $\RCA$.

\begin{prop} \label{rcaToFinite}
$\RCA$ proves $\ST$ for finite graphs.  That is, $\RCA$ implies that every finite graph satisfying condition~(A) has a perfect matching.
\end{prop}
\begin{proof}
Any finite graph $G$ has only finitely many matchings.  
Because being a finite matching is a $\Delta^0_1$ property, there is a maximal such matching by $\Sigma^0_1$ induction.  But since the graph satisfies condition~(A), this matching of maximal support must cover all vertices, as discussed in the proof of Proposition \ref{prop.MAXimpliesPM}.
\end{proof}

The case of $\ST$ for locally finite graphs is more interesting.  Steffens \cite{steffens} points out that this case can be proved directly using the Rado Selection Lemma for such graphs.  That proof can be formalized in $\ACA$ using the version of Rado's theorem in Theorem~III.7.8 of \cite{simpson}.  To consider the additional case of bounded graphs, and to lay a foundation for ideas that will be important later in this paper, we follow a more direct approach.

In particular, the proof below highlights the roles of both condition~(A) and the additional assumption that $G$ is locally finite.   

\begin{definition} 
A locally finite graph is \emph{bounded} provided there is a function $h:V \to \N$ such that for all $x,y \in V$, if $\{x,y\} \in E$ then $h(x) \ge y$.
\end{definition}

\begin{prop} \label{prop-steffens-locfin}
$\ACA$ proves  $\ST$ for locally finite graphs, and $\WKL$ proves $\ST$ for bounded graphs.
\end{prop}
\begin{proof}
  Suppose $G$ is a locally finite graph with vertex set $\N$, and which satisfies condition~(A).  
Define a tree $T \subseteq \N^{<\N}$ by letting $\la a_0, \dots, a_n \ra \in T$ if and only if
    $\{\{0,a_0\},\dots,\{n,a_n\}\}$ is a matching of $\{0,\dots,n\}\cup\{a_0,\dots,a_n\}$. 

Note that any infinite path $\langle a_i :i\in\N \rangle$ through $T$ will correspond exactly to a set of edges $\{i,a_i\}$ that define a perfect matching of $V(G)=\N$.
It is clear that $T$ is definable in $\RCA$, and that $T$ is really is a tree since the matching that witnesses the inclusion of $\tau \in T$ also witnesses the inclusion of any prefix of $\tau$.

Note also that for any string $\langle a_0,\dots,a_n\rangle \in T$, $\{i,a_i\}$ must be an edge in $G$.  
Because $G$ is locally finite, $T$ will be finitely branching, and in the case that $G$ is bounded, $T$ will also be bounded. 

Because $G$ satisfies condition~(A), each imperfect matching of $G$ can be extended to one of greater support. The key use of this property comes in the proof that $T$ is infinite.  
  It suffices to show that for each $n$, there are some vertices $v_0,\ldots,v_n$ and some finite matching $I_n = \{\{0,v_0\},\dots,\{n,v_n\}\}$ of the vertices $F_n = \{0,\dots,n\}\cup\{v_0,\dots,v_n\}$.  We prove this using $\Sigma^0_1$ induction.  
  
For $n = 0$, by condition~(A) applied to the empty matching, there is an $\emptyset$-augmenting path starting at vertex $0$, which just consists of some neighbor of $0$, call it $v_0$.  Then $F_0 = \{0, v_0\}$ has matching $I_0 = \{\{0,v_0\}\}$.

  Now suppose $I_{n} = \{\{0,v_0\}, \{1, v_1\}, \ldots, \{n,v_n\}\}$ is a matching of
  \[
  F_n = \{0, 1, \ldots, n\} \cup \{v_0, v_1, \ldots, v_n\}. 
  \]
  If $n + 1 \in F_{n}$, then we are done (let $v_{n+1}$ be the vertex to which $n+1$ is already matched).  Otherwise, by condition~(A), there is an $I_{n}$-augmenting path $P_n$ beginning at $n+1$.  Because $V(I_{n})$ is finite, the path $P_n$ must be finite, and thus end in some vertex $v_{n+1} \notin F_{n}\cup\{n+1\}$.  Set $F_{n+1} = F_{n} \cup\{n+1,v_{n+1}\}$, and recall that $P_n \mathbin{\Delta} I_n$ is a matching $I_{n+1}$ of $F_{n+1}$.  In this case, note that $I_{n+1}$ will \emph{not} extend $I_{n}$, but \emph{will} extend some matching on the tree.

It follows that $T$ is an infinite, finitely branching tree.  
By $\ACA$ there is an infinite path $P$.  
In the case that $G$ is bounded then $T$ is also bounded, so the existence of this infinite path $P$ follows from $\WKL$ instead.

  All that remains is to verify that the set of edges $M = \{\{i,P(i)\} : i\in \N\}$ defines a perfect matching of $G$.
  Suppose that $M$ is not a matching.  This means some vertex in $G$ is incident to at least two edges of $M$.  Consider the three vertices incident to these two edges.  Now restrict $T$ to the level of $T$ containing all three vertices.  This produces a string in $T$ that does not define a matching, contradicting our definition of $T$.  Clearly $V(M) = \N$, so the matching is perfect.  Thus $M$ is a perfect matching of $G$.
\end{proof}

The upper bounds given above are optimal. 

\begin{prop} \label{wklIsCFB}
  $\ST$ for bounded graphs is equivalent to $\WKL$ over $\RCA$.
\end{prop}

\begin{proof}
  The forward implication is proved in \Cref{prop-steffens-locfin}.  For the reversal, we prove $\WKL$ via $\Sigma^0_1$ separation, which is sufficient by Lemma~IV.4.4 in \cite{simpson}.  Let $f, g:\N \to \N$ be one-to-one functions satisfying $\forall i, j \, [f(i) \ne g(j)]$.  We will define a set $X$ such that $\forall m \, [{f(m) \in X} \wedge {g(m) \notin X}]$.

  For each $n \in \N$, build a disjoint path of odd length (number of edges), in a way that keeps track of the ``center'' edge (for example, by using the evens as the endpoints of the center edges and odds for all other vertices).  At each stage of the construction, add edges to both ends of each path unless $n$ enters the range of $f$ or $g$. In this case, stop building the path after first ensuring that the length is either $1 \pmod 4$ if $n$ is in the range of $f$, or $3 \pmod 4$ if $n$ is in the range of $g$.  If $n$ never appears in the range of either function, build the path forever. 

  The resulting graph will consist of infinitely many disjoint paths, each either a two-way infinite path or else a finite path of odd length.  Such a graph satisfies condition~(A): given a matching $M$ and a vertex $v \notin V(M)$, follow the one or two paths leading away from $v$ for as long as they are $M$-alternating.  If $v$ is on the end of one path, there is one alternating path leading away from it.  In this case, if the alternating path were not augmenting, then it would end in a vertex matched by $M$.  Similarly, if $v$ was in the middle of a path, then the only way neither path leading away from it would be augmenting is if both paths terminate in a vertex matched by $M$.  In both these cases, the path would be even, a contradiction.

  Moreover, the graph is bounded (define $h$ using the effective construction of $G$).  Therefore by  $\ST$ for bounded graphs, there exists a perfect matching of $G$.  Define the set $X$ to be the set of $n$ such that the ``center'' edge in the path for $n$ is in the matching.  If $n$ is in the range of $f$, then the path for $n$ had length $1 \pmod 4$, so the only matching includes the center edge, giving $n \in X$.  On the other hand, if $n$ is in the range of $g$ then the path is length $3 \pmod 4$, so the center edge is not in the matching, giving $n \notin X$.  For any $n$ not in either range, the center edge might or might not be in the matching, which is fine.  Thus $X$ is a separating set as required.
\end{proof}

The full locally finite case was previously shown equivalent to $\ACA$ by Sakakibara \cite{sakakibara}.  We include a proof for completeness.

\begin{prop}[Sakakibara \cite{sakakibara}] \label{acaIsFB}
$\ST$ for locally finite graphs is equivalent to $\ACA$ over $\RCA$.
\end{prop}

\begin{proof}
  The forward implication is proved in \Cref{prop-steffens-locfin}.  
    For the reversal, let $f:\N \to \N$ be a one-to-one function.  We want to show that the range of $f$ exists. 

  Build a graph $G$ consisting of infinitely many disjoint paths, either of length $1$ (edge) or $3$.  Specifically, for each $n$, put an edge $\{4n, 4n + 2\}$ in the graph.  Whenever a number $n$ enters the range of $f$, add edges $\{j, 4n\}$ and $\{4n + 2, k\}$, where $j$ and $k$ are the least unused odd natural numbers.

  Since all vertices of $G$ are part of either a length-$1$ or length-$3$ path, $G$ satisfies condition~(A): for any matching $M$ and vertex $v_0 \notin V(M)$, if $v_0$ is even, then it must be adjacent to an unmatched vertex we can search for and find.  If $v_0$ is odd, then we can search for its even neighbor, which if matched will start a length-$3$ alternating path.  Moreover, $G$ is clearly locally finite.  Thus there is a perfect matching $M$ of $G$.  From this matching, define the range of $f$ as those $n$ for which $\{4n, 4n + 2\} \notin M$.
\end{proof}

Recall that $\Maximality$ implies $\ST$ over $\RCA$, and that $\Maximality$ is a corollary to Steffens' \emph{proof} of $\ST$.  When restricting to locally finite graphs, we will show that these principles are in fact equivalent. 

To show that $\FB$ implies Locally Finite $\Maximality$, we will use the fact that $\FB$ implies $\ACA$, and similarly for $\cFB$ and $\WKL$.  In other words, we prove that every locally finite graph has a maximal matching in $\ACA$, and that this holds in $\WKL$ for bounded graphs. 

\begin{prop} \label{prop.MMforLocFin}
Over $\RCA$, 
  $\Maximality$ for locally finite graphs
is equivalent to 
  $\FB$, 
and 
  $\Maximality$ for bounded graphs
is equivalent to 
  $\cFB$. 
\end{prop}

\begin{proof}
The proof of \Cref{prop.MAXimpliesPM} shows that any maximal matching for a graph satisfying condition~(A) is perfect, so the existence of maximal matchings for locally finite (or bounded) graphs implies  $\ST$ for these graphs.  

For the other direction, we will show that $\ACA$ proves that every locally finite graph has a maximal matching, and $\WKL$ proves that every bounded graph has a maximal matching.
The argument is similar to the proof in Proposition \ref{prop-steffens-locfin}, except now we will skip vertices for which there is no way to extend the matching.

  First, define a sequence $\la b_i \ra_{i < k \le \omega}$ recursively as follows.  Let $b_0$ be the least non-isolated vertex of the graph.  For $n > 0$, let $b_n$ be the least vertex for which there is a matching in $G$ that covers $\{b_0, \ldots, b_n\}$.  The sequence is strictly increasing, and definable over $\ACA$ in general, and in fact $\RCA$ for bounded graphs.  If the sequence $\la b_i \ra$ is finite, then we are done, since any matching that witnesses the last term in the sequence will be a maximal matching.  In the case where the sequence is infinite, proceed as follows.

  Define a tree $T \subseteq \N^{<\N}$ by putting $\langle a_0, a_1, \ldots ,a_n\rangle \in T$ iff $\{\{b_i, a_i\} \st i \le n\}$ is a matching of $\{b_0, \ldots, b_n\} \cup \{a_0,\ldots,a_n\}$.  From the definition of $\la b_i \ra$, for each $i$ there is a string of length $i$ which will be on this tree. Therefore, the tree must be infinite.  For locally finite graphs, the tree is locally finite and definable over $\ACA$, thus in $\ACA$ the tree has a path.  For bounded graphs the tree is bounded and definable over $\RCA$, so has a path in $\WKL$.

  We claim that any path through the tree gives us a maximal matching $M$. Suppose there was a matching $N$ with $V(N) \supset V(M)$.  Let $v$ be the least vertex in $V(N) \setminus V(M)$, and as such $v \ne b_i$ for any $i \in \N$.  Let $j$ be greatest such that $b_j < v$.  Then restricting $N$ to its edges covering $\{b_0, \ldots, b_j, v\}$ gives a finite matching that covers these vertices, which would have put $v = b_{j+1}$, a contradiction.
\end{proof}

\section{Finding matchings in general}
\label{sect.matchings-in-general}
The proofs in Section \ref{sec.locallyfinite} made an essential use of the assumption that $G$ was locally finite, as K\"onig's lemma is false for infinitely branching trees. 

One na{\"i}ve approach to proving  $\ST$ in the general case would be to iteratively use augmenting paths to grow a matching that will cover an increasing number of vertices of the graph.
Unfortunately, there are graphs where this repeated augmentation results in a vertex such that the edge that covers it changes infinitely often, and consequently, that vertex will not be covered in the limit. 

To avoid this obstacle, Steffens' proof in \cite{steffens} centers around building matchings that are stable under augmentation.
An $M$-augmenting path is \emph{proper} provided it passes through an edge of $M$.
It is easy to see that an $M$-augmenting path $P$ is proper if and only if using $P$ to augment $M$ flips an edge in $M$. 

Thus the matchings stable under augmentation are the ones with \emph{no} proper augmenting paths (that is, where augmenting $M$ by any $P$ is equivalent to adding the single edge of $P$ to $M$). 

\begin{definition}[Steffens \cite{steffens}] \label{def.indepmatching}
A \emph{matching $M$ is independent} if there is no proper $M$-augmenting path starting at a vertex $s \in V(G) \setminus V(M)$.  \par
A subgraph $G'$ of $G$ is \emph{independent} if $G'$ has a perfect matching and if every perfect matching of $G'$ is independent.
\end{definition}

Note that the definition of the independence of $M$ is a $\Pi^1_1$ sentence with parameter $G$. 
These definitions are closely related to each other: it is provable in $\RCA$ that a subgraph $G'$ of $G$ is independent if and only if there is at least one independent perfect matching of $G'$ (\Cref{lemmaS2}).

Steffens' proof centers around two key insights.  First, Steffens showed that chains of independent subgraphs of increasing support can  be combined into a single \emph{maximal independent} subgraph of the union.  On its own, this does not prove  $\ST$ or $\Maximality$; it is conceivable that some maximal independent matching might be contained inside a chain of larger, \emph{non}-independent matchings.

Second, Steffens proves a sequence of lemmas that, together, allow one to extend any imperfect maximal independent matching to cover a single unmatched vertex, while also ensuring that the complement of (the subgraph induced by) this new matching still satisfies condition~(A).  By iterating this construction (to cover all unmatched vertices), a perfect matching can then be obtained.

Some of Steffens' lemmas are technical and go beyond the scope of this paper.  However, we give a brief introduction to the results most important for our arguments.  We also include some details to give the reader a sense of the underlying graph theory.  The reader should refer to \cite{steffens} for a complete picture.

To begin, note that the union of two independent subgraphs is \emph{not} always independent.
For example, consider the path of length $2$.  Each of the edges is an independent matching, but the union of those edges is not even matchable (so is clearly not independent).
The following lemma summarizes two ways of combining independent subgraphs that do preserve independence.

\begin{lemma}
\label{lemma.disjointunionindependent}
\label{lemma.indepremoveindep}
The following are provable in $\RCA$.
      (1)    The union of two \emph{disjoint} independent subgraphs is independent.
      (2)    If $I_1$ is an independent subgraph of $G$ and $I_2$ is independent   in $G \setminus I_1$, then $I_1 \cup I_2$ is independent in $G$. 
\end{lemma}

\begin{proof}
  The second property is from Aharoni, Lemma~4.4 in \cite{aharonimatchings}.  Assume that $I_1$ is an independent subgraph of $G$ and $I_2$ is independent in $G \setminus I_1$.  Suppose for a contradiction that $I_1 \cup I_2$ is not independent in $G$.  Note that $I_1 \cup I_2$ has a perfect matching, namely $M = M_1 \cup M_2$, where $M_1$ is a perfect matching of $I_1$ and $M_2$ is a perfect matching of $I_2$.  So assume $M$ is not an independent matching.  Then there is a proper $M$-augmenting path $P$ in $G$ starting at $s \in V(G) \setminus V(M)$.  If $P$ is disjoint from $I_1$, then this contradicts the independence of $I_2$ in $G \setminus I_1$.  So consider $I_1 \cap P$, which must be nonempty.  There must be a vertex $s' \in V(G) \setminus V(M_1)$ adjacent to one of the vertices in this part of the path.  Following $P$ starting at $s'$ cannot be infinite inside $I_1$ because that would be a proper $M_1$-augmenting path, but also cannot leave $I_1$ for the same reason.
  
The first part easily follows from (2).
\end{proof}

For clarity, we will distill the ``graph theoretic core'' of the proof of $\ST$ in \cite{steffens} into the following lemma.

\begin{lemma}[Adapted from \cite{steffens}]
\label{lem-extension} 
\label{lemma.removing-indep-(A)}
The following hold over the given system.
  \begin{enumerate}
    \item $(\RCA)$ If $G$ satisfies condition~(A) and if $M$ is an independent matching of $G$, then $G \setminus V(M)$ satisfies condition~(A).
    \item $(\PtCA)$ Suppose $G$ satisfies condition~(A) and $s \in V(G)$, but there are no nonempty independent subgraphs of $G$.  Then there is a matching $M''$ of $G$ such that $s \in V(M'')$ and $G \setminus V(M'')$ satisfies condition~(A).
  \end{enumerate}
\end{lemma}

To prove $\ST$ using \Cref{lem-extension}, Steffens' proof first obtains a maximal independent matching.  By \Cref{lem-extension} part (1), the graph of the remaining unmatched vertices continues to satisfy condition~(A). Also, by \Cref{lemma.indepremoveindep} part (2), the remaining graph has no nonempty independent subgraphs.  Thus by \Cref{lem-extension} part (2), we can find a matching $M''$ of $G$ that covers any one of the remaining unmatched vertices, and so that its removal preserves condition~(A).  Iterating this process, all vertices of $G$ can be covered by the edges of a matching.
 
Because our statement of Lemma \ref{lem-extension} is not exactly identical to any lemma of Steffens, we close our discussion with a short sketch of its proof, which simply indicates how it follows from the lemmas of Steffens \cite{steffens}.   

\begin{proof}[Proof sketch of \Cref{lem-extension}]
For the first property, let $M'$ be a matching of $G \setminus V(M)$, and let $s \in V(G) \setminus V(M)$ be a vertex not covered by $M'$.  Since $V(M)$ and $V(M')$ are disjoint, $M \cup M'$ is a matching of $G$.  Then, because $G$ satisfies condition~(A), there must be a path $P$ in $G$ that starts at $s$ and augments $M \cup M'$.
    Now it cannot be that $P$ ever enters $M$, for if it did, the restriction of $P$ to $M$ together with the vertex $s'$ before $P$ first enters $M$ and the first vertex of $P$ no longer in $M$ (if there is one) would be a proper $M$-augmenting path.  Since $M$ is independent, this cannot happen.  Thus $P$ is disjoint from $M$, so it is in fact an $M'$-augmenting path contained in $G \setminus V(M)$, as needed.

For the second property, let $M'$ be a maximal (possibly empty) independent matching of $G \setminus s$ (this requires an application of $\MIM$, which, in \Cref{lemma.SteffensLemma3.p12ca}, we will show is provable in $\PtCA$).  Because $G$ satisfies condition~(A), and because $s \notin V(M')$, there is an $M'$-augmenting path $P$ starting at $s$.  Let $M'' = M' \mathbin{\Delta} P$, which is a perfect matching of $V(M') \cup P$.
  Together, Lemmas~5 and 7 of \cite{steffens} are exactly the statement that $G \setminus V(M'')$ satisfies condition~(A). 
  Furthermore, it is straightforward to see that the proofs of Lemmas~5 and 7 of \cite{steffens} can be formalized inside $\ACA$.
\end{proof}

\subsection{Proofs of \texorpdfstring{$\MIM$}{MI} and \texorpdfstring{$\ST$}{PM}}
\label{section.mim}

Given a graph $G$, Steffens' original proof of $\ST$ used infinitely many applications of Zorn's Lemma to find a matching of $G$.
More precisely, Steffens used Zorn's Lemma to prove the existence of maximal independent matchings (our $\MIM$) and then recursively applied that principle $\omega$ times.
In this section, we will first show that $\MIM$ is provable in $\PtCA$ using an inner model technique.  Later, we will use a related line of reasoning to see that it is not possible for either $\ST$ or $\Maximality$ to imply $\PtCA$.

The inner model technique we use is analogous to the ones used in \cite{simpsonkdt} to prove that $\ATR$ implies $\CKDT$ and in \cite{shafer} to prove that $\PCA$ implies Menger's theorem for countable webs.  The primary difference is that while the above proofs were able to use $\omega$- or $\beta$-models; we use $\beta_2$-models because we will need to reflect the existence of a \emph{maximal} independent matching (a $\Sigma^1_2$ property) out of the model, and still have it be maximal independent.
We begin with some background on $\beta$ and $\beta_2$ models.

\begin{definition}[Simpson \cite{simpson}]\label{defn.coded.beta.model}
The following definition is made within $\RCA$.  A \textit{countable coded $\omega$-model} is a set $W \subseteq \N$, viewed as encoding the $L_2$-model $\M = (\N, S_{\M}, +, \cdot, 0, 1, <)$ with $S_{\M} = \{(W)_n : n \in \N\}$.

Let $0 \le k < \omega$.  
A $\beta_k$-model is an $\omega$-model $\M$ such that for all $\Sigma^1_k$ sentences $\varphi$ with parameters from $\M$, $\varphi$ is true if and only if $\M \models \varphi$.
A \textit{countable coded $\beta_k$-model} is a countable coded $\omega$-model $\M$ such that for all $e, m \in \N$ and $X, Y \in S_{\M}$, $\varphi_k(e, m, X, Y)$ is true if and only if $\M \models \varphi_k(e, m, X, Y)$, where $\varphi_k(e, m, X, Y)$ is a universal $\Sigma^1_k$ formula (see \cite{simpson}).  A \emph{countable coded $\beta$-model} is a countable coded $\beta_1$-model.
\end{definition}

To understand the semantic meaning of Definition \ref{defn.coded.beta.model}, let $\mathcal{N} = (\N,S_{\mathcal{N}})$ be any model of $\RCA$.  Then a set $W\in S_{\mathcal{N}}$ defines a countable coded $\beta_k$-model, in the context of the intended model $(\N,S_{\mathcal{N}})$, if for each $\Sigma^1_k$ sentence $\phi_k$ with parameters from $W$, $(\N,S_{\mathcal{N}})\models \phi_k$ if and only if $(\N,\{(W)_n:n\in \N)\models \phi_k$.

By combining Theorem~VII.7.4 and Theorem~VII.6.9(3) of \cite{simpson}, we see that, over $\ACA$, $\PCA$ (resp., $\PtCA$) is equivalent to the statement that for all $X \subseteq \N$, there exists a countable coded $\beta$-model (resp., countable coded $\beta_{2}$-model) $\M$ such that $X \in \M$.

To build a maximal independent matching, we require Lemmas~1 and 2 of \cite{steffens}. 

\begin{lem}[Lemma~1 of Steffens \cite{steffens}] \label{lemmaS1}
Let $G$ be a countable graph. $\RCA$ proves that if $M$ is a perfect matching of $G$, and $G'$ is an independent subgraph of $G$, then there is no edge $\{s, v\} \in M$ such that $s \in V(G) \setminus V(G')$ and $v \in V(G')$.
\end{lem}

\begin{proof}
Suppose for a contradiction that there is an edge $\{s, v\} \in M$ with $s \in V(G) \setminus V(G')$ and $v \in V(G')$, and (by definition of independence of $G'$) let $M'$ be an independent perfect matching of $G'$.  We can now form a proper $M'$-augmenting path starting at $s$, which will contradict the independence of $M'$.  Start with the edge $\{s, v\} \in M$, and let $v_0 = s$ and $v_1 = v$.  By assumption, $v_1 \in V(M')$, so there is an edge in $M'$ from $v_1$ to some $v_2 \neq v_1$.  Note that in $M$, $v_2$ cannot be matched to $v_1$ (because $v_1$ is matched to $v_0 \neq v_2$) and is thus matched to another vertex $v_3$.  Repeating this process, we obtain
a path $P$ that either keeps going forever, or it leaves $G'$.  In either case, it is a proper $M'$-augmenting path.  
Finally, note that the definition of $P$ is $\Delta^0_1$ in $M \oplus M'$, so it exists by $\RCA$.
\end{proof}

\begin{lem}[Lemma~2 of Steffens \cite{steffens}] \label{lemmaS2}
  Let $G$ be a countable graph.  $\RCA$ proves that a subgraph $G'$ of $G$ is independent if and only if there exists an independent perfect matching for $G'$.
\end{lem}

\begin{proof}
The forward direction follows trivially from the definition of independent subgraph.  For the other direction, let $G'$ be a subgraph of $G$ and assume that there exists an independent perfect matching $M'$ for $G'$.  We must show that every perfect matching of $G'$ is also independent.  Suppose toward a contradiction that there is a perfect matching $M$ of $G'$ that is not independent.  Then there is a proper $M$-augmenting path $P$ beginning at some $s \in V(G) \setminus V(G')$, which has second vertex $v \in V(G')$.

Let $v_0 = s$ and $v_1 = v$.  Similar to the proof of \Cref{lemmaS1}, define a path $(v_i)_{i < k \leq \omega}$ such that $\{v_1, v_2\} \in M'$, $\{v_2, v_3\} \in M$, and so forth, alternating between the two matchings.  This is now a proper $M'$-augmenting path starting at $s$, which contradicts the independence of $M'$. 
\end{proof}

We can now prove the first main result of this section.  Note that if $P(G)$ is a property of a graph $G$, then we sometimes write \inM{$P(G)$} as an abbreviation for $\M \models P(G)$.

\begin{thm} \label{lemma.SteffensLemma3.p12ca}
$\PtCA$ proves $\MIM$. 
\end{thm}
\begin{proof}
  By $\PtCA$, there is a countable coded $\beta_2$-model $\M$ containing the graph $G$ as an element.  
  Note that it is arithmetical with parameter $\M$ to define \inM{in\-de\-pen\-dent} matchings, as these are the columns of $\M$ which are matchings that are not properly augmented by  any other column of $\M$.  
  By arithmetical comprehension, using a code for $\M$ as a parameter, we can form the set of all \inM{in\-de\-pen\-dent} matchings in $\M$.  
  Arithmetically relative to this oracle, form a chain of \inM{in\-de\-pen\-dent} matchings $\{M_i\}_{i \in \N}$ in $\M$ by recursion on $n$ as follows.  
  Let $g_n$ denote the $n$-th vertex in $G$.  $M_0$ is the first \inM{in\-de\-pen\-dent} matching that contains $g_0$ if it exists, otherwise it is empty.  If there is an \inM{in\-de\-pen\-dent} matching that extends the support of $M_n$ and contains $g_{n + 1}$, set it to be $M_{n + 1}$.  Otherwise, set $M_{n + 1} = M_n$.  Finally, we form the chain $\{M^*_i\}_{i \in \N}$ by setting $M_0^* = M_0$ and $M^*_{i+1} = \big(M_{i+1} \setminus E(V(M_i))\big) \cup M^*_{i}$.  Then define $M^* = \bigcup_{i \in \N} M^*_i$, as in Steffens' proof.

We claim that each $M_i^*$ is \inM{in\-de\-pen\-dent}.  By definition, each $M_i$ is \inM{in\-de\-pen\-dent}.
By \Cref{lemmaS2}, every perfect matching of $V(M_i)$ is \inM{in\-de\-pen\-dent}, so $V(M_i)$ is an \inM{in\-de\-pen\-dent} subgraph. 
By \Cref{lemma.indepremoveindep}, 
removing an \inM{in\-de\-pen\-dent} subgraph preserves \inM{in\-de\-pen\-dence}, and taking a disjoint union of \inM{in\-de\-pen\-dent} subgraphs (or a union of a chain of \inM{in\-de\-pen\-dent} subgraphs) also preserves \inM{in\-de\-pen\-dence}. 

  We claim that the union $M^*$ is truly independent (outside of $\M$).  If it were not, then there would be a proper $M^*$-augmenting path $P$ starting at a vertex $v \in V(G) \setminus V(M^*)$.  Because the $M_i^*$ form a chain of matchings of increasing support, if $P$ properly augments $M^*$, then for some $i$, the restriction $P_i$ of $P$ to $M_i^*$ (the path $P_i$ is in $\M$ because it is definable from $M_i^*$) will properly augment $M_i^*$.  Because $M_i^*$ exists in $\M$, because $M_i^*$ is augmented by $P_i$ (a $\Sigma^1_1$ property with parameters in $\M$), and because  $\M$ is a $\beta_2$-model (so satisfies $\Sigma^1_2$ reflection), some proper $M_i^*$-augmenting path exists in the model.  But this contradicts the fact that each $M_i^*$ was \inM{in\-de\-pen\-dent}.

  We also claim that $M^*$ is maximal independent.  Suppose it were not, for a contradiction.  Then there would be an independent matching whose vertex set properly contains $V(M^*)$.  Let $g_n$ be the vertex of least index not in $V(M^*)$ that is in the larger independent matching.  Now we have
  \[\exists Y \, [{Y \text{ is independent}} \wedge {V(Y) \supsetneq V(M_{n}^*)} \wedge {g_n \in V(Y)}].\]
Because the statement ``$Y$ is independent'' is $\Pi^1_1$, the above is a $\Sigma^1_2$ sentence.

Since $\M$ is a $\beta_2$-model and the above $\Sigma^1_2$ sentence is true outside of $\M$, it must be true in $\M$.
However, recall that by our definition of the $M_i$, there must not be an \inM{in\-de\-pen\-dent} matching which extends $M_{n}$ and also contains $g_n$ (and $V(M_n) = V(M_n^*)$).  So we have a contradiction, and therefore $M^*$ is a maximal independent matching of $G$.
\end{proof}

One might wonder if the proof above could be simplified by using a less complex 
(than $\Pi^1_1$) way of saying that a matching is independent.  However, using an argument similar to those in \Cref{sect.lower-bounds-general}, it is easy to see that deciding whether a given matching in a computable graph is independent is $\Pi^1_1$-complete.

It is now possible to give an upper bound on the strength of $\ST$. 
The system $\PtCA^+$ is the system that permits $\omega$-many iterated applications of $\PtCA$. 

\begin{theorem}
  \label{theorem.steffens-in-Pi12-CA+}
  $\PtCA^+$ proves $\ST$.  
\end{theorem}

\begin{proof} The idea of the proof is the same as Steffens'.  Let $G$ be a graph that satisfies condition~(A) and suppose $V(G) = \{v_0, v_1,\ldots\}$.  Apply $\MIM$ to obtain a maximal independent matching $M_0$ of $G$.  Define $G_1 = G \setminus V(M_0)$, which continues to satisfy condition~(A) by \Cref{lemma.removing-indep-(A)}.  Apply the second part of Lemma~\ref{lem-extension} (which entails another application of $\MIM$) to obtain a matching $M_1$ contained in $G_1$ which matches the vertex 
 of least index not matched by $M_0$, such that $G_2 = G_1 \setminus V(M_1)$ still satisfies condition~(A).  At this point, we do not know whether $G_2$ has nontrivial independent subgraphs (or unmatched vertices), so we start the process over again with $G_2$ as the new $G$.

Iterate the above process to obtain a sequence of matchings $\{M_n\}_{n\in \N}$, with pairwise disjoint vertex sets, such that every vertex of $G$ is matched by some $M_i$.  We see that $M = \bigcup_{n \in \N} M_n$ is the desired perfect matching of $G$.  Notice that we have potentially applied $\MIM$ $\omega$-many times in the above construction.
\end{proof}

This proof appears to do more work than is actually needed.  Because $\ST$ is \emph{true}, the graph $G$ has a perfect matching $M$.  Because $M$ covers each vertex of $G$, it trivially satisfies the definition of an independent matching, so $M$ is an independent matching of maximal support.  But that means that, as long as $\ST$ is true, the first application of $\MIM$ in its proof already yields the perfect matching of $G$.  The rest of the construction (which involved finding infinitely many new maximal independent matchings) was, \textit{in retrospect}, unnecessary.  
We will return to this observation in Section \ref{sect.ProofOfMAX}.

Indeed, the complexity of $\ST$ means that this upper bound cannot be sharp. 
We will use a generalization of a well known fact, following the presentation of  Marcone \cite{marcone}.
Although we will only need the cases for $k=1$ and $k=2$ in this paper, we include a general statement of the property for completeness.

\begin{prop}\label{doesnotimply}
Let $k \geq 1$.  
No $\Pi^1_{k+1}$ statement that is consistent with $\ATR$ can imply $\Pi^1_k$-$\mathsf{CA}_0$, even over $\ATR$.
\end{prop}
\begin{proof}
Consider the sentence $\forall X \psi(X)$, where  $\psi(X)$ is $\Sigma^1_k$ and suppose that the theory $T$ consisting of $\ATR$ and $\forall X \psi(X)$ is consistent.

Suppose toward a contradiction that $\forall X \psi(X)$ implies $\Pi^1_k$-$\mathsf{CA}_0$ over $\ATR$. 
Then $T$ proves the existence of a countable-coded $\beta_k$-model $\mathcal M$. 
By definition, $\mathcal M\models \psi(X)$ for each $X\in \mathcal M$, so $\mathcal M\models \forall X\psi(X)$. Since every $\beta$-model (and so every $\beta_k$-model) is a model of $\ATR$, we have that $\mathcal M$ is a model of $T$.

But by G\"odel's Second Incompleteness Theorem, no consistent theory $T$ which implies $\PCA$ can prove its own consistency, contradicting the assumption that $\forall X \psi(X)$ implies $\Pi^1_k$-$\mathsf{CA}_0$ over $\ATR$.
\end{proof}

Recall that the statement of $\ST$ itself has the form 
\[\forall X \, \big[(\forall Y_1) (\exists Y_2) \theta(X, Y_1, Y_2) \rightarrow \exists Z \, \psi(X, Z)\big],\]
with $\theta, \psi$ arithmetical, and so is equivalent (after prenexing) to a $\Pi^1_3$ formula.  Also, since $\PtCA^+$ implies $\ST$, it is consistent with $\ATR$.  Thus we get the following immediately from \Cref{doesnotimply}.

\begin{cor}\label{theorem.steffens-not-imply-Pi12-CA}
   $\ST$ does not imply $\PtCA$ over $\RCA$. 
\end{cor}

\subsection{A proof of \texorpdfstring{$\Maximality$}{MM}}

\label{sect.ProofOfMAX}

We now turn our attention to the strongest version of $\ST$ studied in this paper, the statement that \emph{every} graph has a matching of maximal support.
To prove this, we will need another lemma of Steffens.
\begin{lem}[Lemma 5 of Steffens \cite{steffens}] \label{lemma-s45}
The following is provable in $\ACA$.   A graph $G$ satisfies condition~(A) iff for every independent subgraph $G'$ of $G$ and for every vertex $s \in V(G) \setminus V(G')$, there exists a vertex $v \in V(G) \setminus V(G')$ such that $\{s, v\} \in E(G)$.
\end{lem}
      
\begin{proof}
It is straightforward to check that the proofs of Steffens' Lemmas~4 and 5 from \cite{steffens}, the latter relying on the former, both hold in $\ACA$.
\end{proof}

\begin{theorem}
\label{PMplusMIMimpliesMAX}
$\MIM + \ST$ implies $\Maximality$ over $\RCA$.
\end{theorem}  
\begin{proof}  
  Fix a countable graph $G$.  Applying $\MIM$, 
  we obtain a maximal independent subgraph $I \subseteq G$ and a corresponding independent matching $M$.

  If $I = G$, then we have a perfect (hence maximal) matching and we are done.  So suppose $I \neq G$, and let $N \subseteq V(G) \setminus V(I)$ be the set of vertices with no neighbors in $G \setminus I$.  (The vertices in $N$ are not necessarily isolated in $G$ as they may have neighbors in $I$ itself.)  Consider the subgraph $H = V(G) \setminus (V(I) \cup N)$.  We claim that $H$ does not have any nonempty independent subgraphs.  Indeed, since any nonempty independent subgraph $G'$ of $H$ is also a nonempty independent subgraph of $G \setminus I$, $G'$ could therefore be combined with $I$ to form a new, larger independent subgraph, by \Cref{lemma.disjointunionindependent} part 2, which would contradict the maximality of $I$.  Therefore the only independent subgraph of  $H$ is the empty subgraph, and clearly each vertex in $H$ has a neighbor in $H$, so $H$ satisfies condition~(A) by \Cref{lemma-s45}.

  Applying  $\ST$, we obtain a perfect matching $\hat M$ of $H$.  First, note that $M \cup \hat M$ is a matching because $V(M)$ and $V(\hat M)$ are disjoint sets of vertices.  We wish to show that $M \cup \hat M$ is a matching of $G$ of maximal support.  The only vertices not covered by it are vertices in $N$.  And notice, neighbors of vertices in $N$ can only be in $I$, so no matching that covers $I = V(M)$ can cover any vertex in $N$, by \Cref{lemmaS1}.
Thus, $M \cup \hat M$ must be a maximal matching of $G$.
\end{proof}
 
Theorem \ref{PMplusMIMimpliesMAX} is important for two reasons.  
First, it allows us to use the upper bounds for $\ST$ and $\MIM$ to obtain an upper bound for $\Maximality$. 
Second, and more surprisingly, the fact that $\Maximality$ has lower sentence complexity than $\MIM$ will allow us to separate both principles from $\PtCA$. 

\begin{corollary} \label{proofofMAX}
$\PtCA^+$ proves $\Maximality$, and $\Maximality$ does not imply $\PtCA$ over $\RCA$.
\end{corollary} 

\begin{proof}
$\PtCA^+$ proves $\ST$ by \Cref{theorem.steffens-in-Pi12-CA+}, and $\PtCA$ proves $\MIM$ by \Cref{lemma.SteffensLemma3.p12ca}.  Thus, $\PtCA^+$ proves $\Maximality$ by \Cref{PMplusMIMimpliesMAX}.

Note that $\Maximality$ is the following $\Pi^1_3$ sentence: for each $G$, there is an $M$ such that for any other $M'$, if $M$ and $M'$ are both matchings, then $V(M)$ is not strictly contained in $V(M')$. Therefore, since $\Maximality$ is true and thus consistent with $\ATR$, $\Maximality$ does not imply $\PtCA$ by \Cref{doesnotimply}.
\end{proof}

Because maximal matchings are independent, the upper bound on $\Maximality$ also applies to $\MIM$. 

\begin{proposition} \label{cor.MAXimpliesMIM}
$\Maximality$ implies $\MIM$ over $\RCA$.
\end{proposition}

\begin{proof}
Any maximal matching $M$ of a graph $G$ must be independent.  To see why, consider any $M$-augmenting path $P$.  If $P$ were a proper augmenting path, then $M \mathbin{\Delta} P$ would be a matching with strictly greater support than $M$, which contradicts the fact that $M$ is a maximal matching. 
Because every independent matching is also a matching, the existence of matchings of maximal support implies the existence of maximal independent matchings.
\end{proof}

\begin{corollary}
$\MIM$ does not imply $\PtCA$ over $\RCA$.
\end{corollary}
\begin{proof}
By Proposition \ref{cor.MAXimpliesMIM}, $\Maximality$ implies $\MIM$ over $\RCA$.  Since $\Maximality$ does not imply $\PtCA$,  it follows that $\MIM$ cannot imply $\PtCA$. 
\end{proof}

It is not clear whether any of $\Maximality$, $\MIM$, or $\ST$ are equivalent over $\RCA$. 
To shed light on this, note that 
the proof of $\Maximality$ required \emph{both} $\ST$ \emph{and} $\MIM$.  However, only a very special case of $\ST$ is used in the proof.  It is natural to ask {exactly} how much of  $\ST$ is actually used in the proof of $\Maximality$.
 
\begin{statement} \label{defn.star}
Let $(\star)$ denote the statement ``if $G$ satisfies condition~(A) and $G$ has no nontrivial independent subgraphs, then $G$ is empty.''
\end{statement}

Note that $(\star)$ is exactly the statement that the first maximal independent matching obtained in the proof of $\ST$ is a perfect matching of the graph, since it asserts that $G \setminus V(M) = \emptyset$. 
This tells us immediately that $\MIM + (\star)$ implies $\ST$.

Similarly, in the proof of \Cref{PMplusMIMimpliesMAX}, $(\star)$ can be used to prove that the maximal independent matching is itself maximal.  
To see why, recall that the subgraph $H$ from the proof of \Cref{PMplusMIMimpliesMAX} has no nonempty independent subgraphs and has no isolated vertices, so satisfies condition~(A) by \Cref{lemma-s45}.
By $(\star)$, $H$ is the empty subgraph so its perfect matching $\hat M$ exists (and is empty), as needed in the proof  of \Cref{PMplusMIMimpliesMAX}.  
This yields the following corollary. 

\begin{corollary}
   $\MIM + (\star)$ implies $\Maximality$ over $\RCA$.
\end{corollary}

In fact, $(\star)$ is a special case of $\ST$.
 
\begin{proposition}
  $\ST$ implies $(\star)$ over $\RCA$.   
\end{proposition} 
\begin{proof} 
  Let $G$ be a graph which satisfies condition~(A) and has no nonempty independent subgraphs.  By $\ST$, $G$ has a perfect matching $M$.  In other words, $V(M)=V(G)$.  In addition, note that $M$ is maximal, so $M$ is an independent matching by the proof of \Cref{cor.MAXimpliesMIM}, and so $V(M)$ is an independent subgraph of $G$.  By assumption, $G$ has no nonempty independent matchings, and so we must have $V(M)=\emptyset$.  
Because $V(M)=V(G)$,  $G$ must be empty.
\end{proof}

On the face of it, $(\star)$ is a straightforward, true principle.  Thus, it is reasonable to conjecture that it has a proof in $\PCA$. We will show in Theorem \ref{thm.MIM.implies.PCA} that $\MIM$ implies $\PCA$, so such a proof of $(\star)$ would imply the equivalence of $\MIM$ and $\Maximality$. 
In addition, if $(\star)$ is provable in $\PCA$, then it may be possible to prove $\ST$ by iterating a contrapositive version of $(\star)$ transfinitely many times. 
  
\begin{conjecture}
$\Maximality$ and $\MIM$ are equivalent over $\RCA$, 
$\PTR$ proves $\ST$,
and $\PtCA$ proves $\Maximality$.
\end{conjecture}

Although apparently simple, $(\star)$ is surprisingly difficult to either prove or to code into.  
No direct proof of $(\star)$ is known to the authors, other than its proof from $\ST$.
\begin{question} 
What is the exact strength of $(\star)$?
Does $(\star)$ imply $\ST$, or can $(\star)$ be separated  from $\ST$ using a forcing construction?
\end{question}

Recall that $\MIM+(\star)$ implies $\ST$.  Thus even a proof of $(\star)$ in $\PtCA$ would yield an improved upper bound on $\ST$ and hence on $\Maximality$.

\section{Lower bounds for matchings in general}
\label{sect.lower-bounds-general}

In this section, we establish lower bounds on the complexity of the principles whose upper bounds were given above.  

In many of these reversals, we uniformly convert a tree $T$ into what we will call its \textit{doubling tree} $\hat T$.  The idea is that each vertex $v$ in $T$ \emph{except} the root gets replaced by two vertices, connected by an edge (we call this a \textit{doubling edge}).  The bottom vertex of that edge is a child of the parent of $v$; the top vertex of the edge is the parent of every child of $v$.  Note that Aharoni, Magidor, and Shore \cite{aharoni_magidor} also construct this type of tree (in their Theorem~4.13).
Clearly $T$ has an infinite path if and only if $\hat T$ has an infinite path.  
More importantly, there is a correspondence between paths in $T$, matchings of $\hat{T}$, and condition~(A).

\begin{lem} \label{pathIffMatch}
  Let $T \subseteq \N^{<\N}$ be a tree, and let $\hat T$ be its doubling tree.  Then the following are equivalent over $\RCA$:
  \begin{enumerate}
    \item $T$ has an infinite path. 
    \item $\hat T$ has a perfect matching. 
    \item $\hat T$ satisfies condition~(A).
   \end{enumerate}
\end{lem}

\begin{proof}
$(1 \rightarrow 2)$.  Suppose there is an infinite path $P$ through $T$.  We define a perfect matching $M$ of $\hat T$ as follows.  For every edge in $P$, put the corresponding edge from $\hat T$ into $M$.  $M$ is a matching since the corresponding path in $\hat T$ alternates between edges in $T$ and the added doubling edges.
 For each vertex not on $P$, include its doubling edge in $M$ (the root is on $P$, so every vertex not on $P$ is adjacent to exactly one doubling edge).  Note that no two doubling edges are adjacent in $\hat T$. Now every vertex in $\hat T$ is matched by $M$, making it a perfect matching.
 
 $(2 \rightarrow 3)$.  This holds in $\RCA$ for all graphs (\Cref{prop-perfect-to-A}).
 
 $(3 \rightarrow 1)$. Let $M$ be the matching consisting of exactly the doubling edges in $\hat T$.  This leaves precisely the root $r$ of $\hat T$ unmatched.  By condition~(A), there is an $M$-augmenting path starting at $r$.  Since no other vertex is unmatched, this path must be infinite, and corresponds to an infinite path back in $T$.
\end{proof}

\subsection{Lower bounds for \texorpdfstring{$\ST$}{PM} }

\begin{theorem} \label{StefToS11AC}
  $\ST$ implies $\SAC$ over $\RCA$.
\end{theorem}

\begin{proof}
In \Cref{acaIsFB}, we proved that the restriction of $\ST$ to locally finite graphs is equivalent to $\ACA$, so we may work over $\ACA$.  Let $\langle T_k : k\in \N\rangle$ be a sequence of trees $T_k \subseteq \N^{<\N}$ such that $\forall k \, [T_k \mbox{ has a path}]$.  Let $\la \hat T_k : k\in \N\ra$ be the associated sequence of doubling trees, and define $G = \bigsqcup_{k \in \mathbb{N}}{\hat T_k}$.  By Theorem~V.1.7$^\prime$ of \cite{simpson}, it suffices to show there exists a sequence $\langle g_k : k \in \mathbb{N} \rangle$ so that
  \[\forall k \, [g_k \text{ is a path through } T_k].\]

  We claim that $G$ satisfies condition~(A).  In the proof of \Cref{pathIffMatch}, we saw that $\hat T_i$ satisfies condition~(A) if and only if $ T_i$ has an infinite path.
  Since each $ T_i$ has an infinite path, $\hat T_i$ satisfies condition~(A) for each $i \in \N$.  It is easy to see, working in $\ACA$, that a disjoint union of graphs satisfying condition~(A) also satisfies condition~(A), and therefore $G$ must satisfy condition~(A).  This proves the claim.

  Then by $\ST$, there is a perfect matching $M$ of $G$.  
  This allows us to uniformly define an infinite path $P_k$ through the doubling tree $\hat T_k$ as follows.  Starting at the root, follow the matching up the tree.  The root vertex is matched by $M$ to exactly one child.  That vertex has only one child, via its doubling edge.  This child is matched by $M$ to exactly one of its children, and so on.  
  By restricting each $P_k$ to $T_k$, we obtain a sequence $\langle g_k : k \in \mathbb{N} \rangle$, where $g_k$ is a path through $T_k$ for each $k \in \mathbb{N}$. 
\end{proof}

In the proof above, we do not appear to use the full strength of $\ST$. 
As we have already noted, some of the complexity of $\ST$ appears to arise from the complexity of deciding if condition~(A) holds.  
But in the proof above, it was computable to show that each $\hat T_i$ satisfied condition~(A), and arithmetical to show that the union of a sequence of pairwise disjoint graphs, all of which satisfy condition~(A), satisfies condition~(A).

It is important to note that deciding whether condition~(A) holds for countable graphs is much harder in general.  Indeed, deciding whether a computable graph satisfies condition~(A) is $\Sigma^1_1$-hard.

\begin{proposition}
The set of indices of computable graphs which satisfy condition (A) is $\Sigma^1_1$-hard.
\end{proposition}

\begin{proof}
By Theorem 16.XX in \cite{rogers}, it suffices to show that given a computable tree $T$, we can find a computable graph $G$, such that $T$ has an infinite path if and only if $G$ satisfies condition (A).
Given $T$, let $G = \widehat T$, the doubling tree of $T$.
By \Cref{pathIffMatch} we have that $\widehat T$ satisfies condition (A) if and only if $T$ has an infinite path, and we are done.
\end{proof}

While condition~(A) has a $\Pi^1_2$ definition, it cannot be $\Pi^1_2$-complete. 
To see why, recall the statement of $\ST$: $G$ satisfies condition~(A) if and only if $G$ has a perfect matching.  
Since deciding whether $G$ has a perfect matching is $\Sigma^1_1$, it is thus possible to decide if condition~(A) holds using a $\Sigma^1_1$ statement.  
 
For at least some classes of graphs, condition~(A) is {equivalent over $\RCA$} to a sentence of simpler complexity.  
In the special case of doubling tree graphs, \Cref{pathIffMatch} says that $\RCA$ proves the equivalence between satisfying condition~(A) and having a perfect matching.
Since the disjoint union of graphs satisfying condition~(A) continues to satisfy condition~(A), the restriction of  $\ST$ to disjoint unions of doubling trees is actually $\Pi^1_2$. 

More generally, the reversal from $\ST$ to $\SAC$ can be seen as using only the statement that the union of a sequence of pairwise disjoint graphs, all of which have perfect matchings, has a perfect matching.  
This can be viewed as a $\Pi^1_2$ version of $\ST$, 
which we refer to as $\CPM$.

\begin{corollary} \label{disjointunion}
The following are equivalent over $\RCA$:
  \begin{enumerate}
    \item $\SAC$
    \item ($\CPM$) Fix a sequence of pairwise disjoint graphs  $\la G_i \ra$.  
    If $G_i$ has a perfect matching for each $i$,
    then $G = \bigsqcup G_i$ has a perfect matching.
  \end{enumerate}
\end{corollary}

\begin{proof} $(1 \rightarrow 2)$ is an easy application of $\SAC$.  $(2 \rightarrow 1)$ is a consequence of the proof of \Cref{StefToS11AC}.
\end{proof}

Of course, Steffens selected condition~(A) for its graph-theoretic content, rather than for its formula complexity.

\begin{question}
Is there a natural $\Delta^1_2$ statement equivalent to condition~(A), which does not involve simply checking for the existence of a perfect matching?
\end{question}
If this is the case, it would yield a version of full $\ST$ whose complexity would be $\Pi^1_2$, and hence this version of $\ST$ would not imply $\PCA$, even over $\ATR$  (Proposition~4.17 from \cite{aharoni_magidor}).

\subsection{The strength of \texorpdfstring{$\MIM$}{MIM} and Sequential \texorpdfstring{$\ST$}{PM}}

\begin{theorem} \label{thm.MIM.implies.PCA}
  $\MIM$ implies $\PCA$ over $\RCA$.
\end{theorem}
 
\begin{proof}
First, we claim that $\MIM$ implies $\ACA$.  Using the same construction as in the proof of \Cref{acaIsFB}, we get a graph consisting of disjoint paths of length $1$ or $3$.  Since each disjoint path has as its maximal independent matching a perfect matching, the maximal independent matching of the entire graph will be its perfect matching, from which we can define the range of the given function.

We may now work over $\ACA$.   
Let $\langle T_i : i \in \mathbb{N} \rangle$ be a sequence of trees $T_i \subseteq \N^{<\N}$.   
We wish to define a set $Z$ such that $i \in Z$ if and only if $T_i$ has an infinite path.
Form the associated sequence $\langle \hat T_i : i \in \mathbb{N}\rangle$ of doubling trees.  
Let $G$ be the disjoint union $G = \bigsqcup \hat T_i$, and by $\MIM$, let $M$ be a maximal independent matching for $G$.

Consider $M$ restricted to $\hat T_i$.  Note that if this is not a perfect matching, then $\hat T_i$ cannot have any perfect matching, since such a perfect matching would be an independent matching of larger support.  Thus $M$ restricted to $\hat T_i$ is a perfect matching if and only if $\hat T_i$ has a perfect matching, which by Lemma \ref{pathIffMatch} occurs if and only if $T_i$ has an infinite path.  

Since $\ACA$ can form the set of $i$ such that $M$ restricted to $\hat T_i$ is a perfect matching of $\hat T_i$, we get the desired set $Z$. 
\end{proof}

Recall from Statement \ref{statements}, Sequential $\ST$ asserts that for each sequence of disjoint graphs $\la G_i : i \in \N \ra$, there is a sequence of matchings $\la M_i : i\in \N\ra$ such that for all $i$, if $G_i$ satisfies condition~(A), then $M_i$ is a perfect matching of $G_i$.

Sequential $\ST$  clearly implies $\ST$ over $\RCA$, and it is natural to conjecture that Sequential $\ST$ is strictly stronger than $\ST$. 
Although less common in reverse mathematics, principles like Sequential $\ST$ sometimes occur in the study of Weihrauch principles.  There, Sequential $\ST$ is called the ``parallelization'' of the ``total continuation'' of $\ST$ (see \cite{kihara}, Section 8, for another similar principle).

\begin{prop} \label{prop.MaxImpliesSequential}
  $\Maximality$ implies Sequential $\ST$ over $\RCA$.
\end{prop}

\begin{proof}
  Assume $\Maximality$, and let $S = \la G_i : i \in \N \ra$ be a sequence of pairwise disjoint graphs.  Let $G' = \bigoplus G_i$ be the effective disjoint union of the columns of $S$, and apply $\Maximality$ to obtain a maximal matching $M$ of $G'$.  For each $i$, let $M_i$ be the restriction of $M$ to $G_i$.  To verify that $\la M_i : i \in \mathbb{N} \ra$ is the desired sequence of edge sets, fix $i$ and suppose $G_i = (V, E)$ satisfies condition~(A).  For a contradiction, assume $M_i$ is not a perfect matching of $G_i$.  Then there is a vertex $s \in V \setminus V(M_i)$.  Since $G_i$ satisfies condition~(A), there is an $M_i$-augmenting path $P$ starting at $s$.  So $M_i \mathbin{\Delta} P$ is a matching of $G_i$ that improves the support of $M_i$, contradicting the fact that $M$ is maximal.
\end{proof}

In fact, our reversal from $\MIM$ to $\PCA$ also goes through in Sequential $\ST$. 

\begin{prop} \label{seqToPca}
  Sequential $\ST$ implies $\PCA$ over $\RCA$. 
\end{prop}

\begin{proof}
As Sequential $\ST$ implies $\ST$ which implies $\ACA$ (all over $\RCA$), we may work over $\ACA$.
Let $\la T_i : i \in \N \ra$ be a sequence of trees.  It suffices
to show that the set $I = \{i : T_i \text{ has an infinite path} \}$ exists.

For each tree $T_i$, consider its doubling tree $\hat T_i$.  Applying  Sequential  $\ST$ to $\la \hat T_i : i \in \N \ra$, we obtain a sequence  $\la M_i : i \in \mathbb{N} \ra$ of edge sets with the property that if $\hat T_i$ satisfies condition~(A), then $M_i$ is a perfect matching of $\hat T_i$.
  
 Since it is arithmetical to decide if a given $M_i$ is a perfect matching, we can form $\{i \st M_i \text{ is a perfect matching of } \hat T_i\}$ in $\ACA$.  By \Cref{pathIffMatch}, this is the set of $i$ such that $T_i$ has an infinite path.
\end{proof}

In the previous section, we showed that $\MIM$ is provable in $\PtCA$, that $\Maximality$ is provable in $\PtCA^+$, and conjectured that $\Maximality$ is provable in $\PtCA$. 

\begin{question}
What is the exact strength of $\Maximality$ and, separately, of $\MIM$?
\end{question}

\section{Matchings for graphs with no infinite paths}\label{sect.noinfpaths}
In this final section, we will refine our analysis of $\ST$ by considering a special case, whose strength is in the region of $\ATR$.
We write $\FPST$ to denote the restriction of $\ST$ to graphs with only finite paths. 

Recall that much of the complexity of $\ST$ comes from the complexity of condition~(A).
In the case of $\FPST$, condition~(A) is equivalent to a $\Pi^1_1$  formula (with the graph as a parameter), and $\FPST$ is equivalent to a $\Pi^1_2$ sentence.   Moreover, the notion of ``independent matching'' becomes arithmetical.
These reductions in complexity result in significantly improved upper bounds.

\begin{theorem}
\label{steffensnoinf} 
$\PCA$ proves $\MIM$ for graphs without infinite paths.
\end{theorem}

\begin{proof} Let $\M$ be a countable coded $\beta$-model containing $G$.  Using $\ACA$ and a code for $\M$ as a parameter, construct a maximal independent matching as described in our proof of full $\MIM$ (\Cref{theorem.steffens-in-Pi12-CA+}) which used a $\beta_2$-model.  The key difference is that now the property of a matching being independent is arithmetical in the matching and the graph, and therefore the formula witnessing a counterexample to a matching being maximal independent is $\Sigma^1_1$ (with parameters from the model), and therefore reflects into or out of a $\beta$-model.  The rest of the proof is analogous.
\end{proof}

\begin{corollary}\label{fpstinpi11ca+}
$\PCAp$ proves  $\FPST$.  
\end{corollary}
\begin{proof}
Because there are no infinite paths in $G$, the property of a matching being independent is arithmetical and the property of a graph satisfying condition~(A) is $\Pi^1_1$.  Therefore the proof is analogous to the proof given earlier of $\ST$ in \Cref{theorem.steffens-in-Pi12-CA+}, this time using countable coded $\beta$-models.
\end{proof}

Because $\FPST$ is equivalent to a $\Pi^1_2$ sentence,  \Cref{doesnotimply} with $k=1$ shows that this upper bound is not optimal. 

\begin{corollary}\label{fpstinpi11ca}
$\FPST$ does not imply $\PCA$ over $\RCA$. 
\end{corollary}

\subsection{Perfect matchings code hyperarithmetical sets}

To obtain lower bounds, we will code hyperarithmetical sets into $\FPST$.  
We assume a basic knowledge of hyperarithmetical theory, and refer the reader to Ash and Knight~\cite{ashknight} for additional background.

We will identify each computable ordinal with its corresponding notation in Kleene's $\O$.  
For each $X \subseteq \N$ and $e \in \O^X$, the set $H_e^X$  (referred to as an \emph{$H^X$-set}) corresponds to the hierarchy resulting from recursively taking jumps of $X$ along the $X$-computable well order with notation $e$.  
As pointed out by Aharoni, Magidor and Shore in \cite{aharoni_magidor}, $\ATR$ is equivalent to
\[\forall X \, \forall e \, [e \in \O^X \rightarrow H_e^X \text{ exists}].\]

We will show that for each set $X$ and for each $e \in \O^X$, there is an $X$-computable graph $G^X_e$ which satisfies condition~(A), such that any perfect matching of $G^X_e$ computes the set $H^X_e$.
As a result of the  uniformity of our construction, it follows that $\mathsf{ATR}$ is Weihrauch reducible to $\ST$ (see \Cref{sect.weihrauch}).  
The $\Pi^1_2$ nature of condition~(A) complicates the reverse-mathematical picture, which will be discussed at the end of the section.  

We follow the general approach of \cite{aharoni_magidor}, by   defining graphs whose perfect matchings code the truth or falsity of \emph{propositions}.  
We also recursively show how, given sentences whose truth is coded by the perfect matchings of coding graphs, to define a new graph that codes the negation, conjunction, or quantification of these sentences. 
We will then be able to code the relation $n \in X$ as a simple $X$-computable proposition.  
By coding negations, conjunctions, and existentials of already coded propositions, we will code the membership relation of the $H^X$-sets corresponding to successor ordinals.
Finally, by coding sequences of already coded propositions, we will code the membership relation of the $H^X$-sets corresponding to limit ordinals.
Appealing to effective transfinite recursion, a consequence of the recursion theorem, we obtain a single uniform computation that maps sets $X$ and codes $e$ to graphs $G_e^X$.

\begin{definition}
  A \emph{coding graph}, illustrated in \Cref{figure.basiccodinggraph},
  is a tuple $\la G, l, r, c \ra$ such that $G$ is a connected graph and $l, r, c \in V(G)$.  We write   $\inside{G} = G \setminus \{l, r, c\}$ to refer to the ``interior'' of the coding graph.  We also require that:
  $r$ is adjacent to a unique vertex in $V(\inside{G})$,  $l$ is adjacent to a unique vertex in $V(\inside{G})$, and $c$ is adjacent to $l$ and $r$, but is not adjacent to any vertices in $V(\inside{G})$.
\end{definition}

\begin{figure}[htb]
  \-\hfill
    \begin{tikzpicture}
      \draw (0,0.5)--(-1.5,0)--(0,-0.5)--(1.5,0)--(0,0.5);
      \node[fill=white,shape=rectangle,draw,inner xsep=4mm,inner ysep=2mm,dashed,rounded corners=3mm] at (0,0.5) {$\inside{G}$};
      \node[fill=white,shape=circle,draw] at (-1.5,0) {$l$};
      \node[fill=white,shape=circle,draw] at (1.5,0) {$r$};
      \node[fill=white,shape=circle,draw] at (0,-0.5) {$c$};
    \end{tikzpicture}
    \hfill\-
  \caption{The coding graph $\la G, l, r, c \ra$}
  \label{figure.basiccodinggraph}
\end{figure}
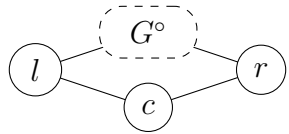

We will prove that each of the coding graphs that we recursively define has a perfect matching, and hence satisfies condition~(A).  In fact, we will prove that each graph we build will have a \emph{unique} perfect matching.

\begin{definition}
  Suppose that a coding graph $\la G, l, r, c \ra$ has a unique perfect matching.
  We say that this graph \emph{codes true} if the unique matching contains the edge between $\inside{G}$ and $l$, and that it \emph{codes false} if the unique matching contains the edge between $\inside{G}$ and $r$.
\end{definition}

The central  vertex  $c$ is  included to ensure that the coding graph will have a matching.
In each step of the recursive construction of coding graphs, we will modify the coding graphs from the previous stage, both by removing their central vertices and by adding new vertices and edges.

\begin{lemma}\label{lem.codingtruenegation}
\-\\
(1) There is a coding graph with a unique matching that codes true.
\\
(2) There is a single computable procedure which, given any coding graph with a unique matching that codes a predicate $P$, returns another coding graph with a unique matching that codes the negation of $P$.
\end{lemma}
\begin{proof}
  (1) To code ``true,'' we define a coding graph $\la G, l, r, c \ra$ with three additional vertices $\{x,y,z\}$, and edges consisting of the cycle $x$-$y$-$r$-$c$-$l$-$x$, together with the single edge $\{y, z\}$.  
This graph is illustrated in \Cref{figure.codingTrueNeg}, left.
Note that the interior of this graph has vertex set $V(\inside{G}) = \{x, y, z \}$.  
  Starting with the edge $\{y,z\}$, it is easy to check that this graph has a unique perfect matching.  

  (2) Suppose we are given a coding graph $\la G, l, r, c\ra$ coding $P$.  We define a coding graph $\la G', l', r', c\ra$ that codes $\neg P$ by setting $\inside{G'} = G \setminus \{c\}$, connecting $l$ to $l'$, and connecting $r$ to $r'$.  This graph is illustrated in \Cref{figure.codingTrueNeg}, right.

  By assumption, $\la G, l, r, c\ra$ has a unique perfect matching.  To see why the new graph $\la G', l', r', c\ra$ has a {unique} perfect matching, note that in any perfect matching, exactly one of $l$ or $r$ matches into $\inside{G}$.  Without loss of generality, assume that $l$ matches into $\inside{G}$.  Then $r$ needs to be matched to $r'$, and $c$ needs to be matched to $l'$, thus yielding a unique perfect matching.

  Finally, we show that this new graph correctly codes $\neg P$. If the original graph $G$ codes true, then the unique matching of $G'$ must match $l$ into $\inside{G}$, which forces $\{ r, r'\}$ to be in the matching, so $G'$ codes false as desired.  On the other hand, if $G$ codes false, then a similar argument shows $G'$ codes true.
\end{proof}

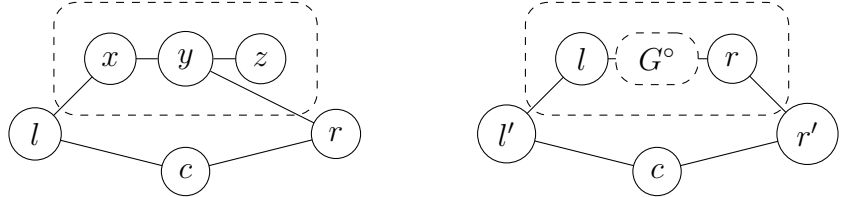
\begin{figure}[htb]
  \-\hfill
  \begin{tikzpicture}
    \draw[dashed,rounded corners=3mm] (-1.75,0.25) rectangle (1.75,1.75);
    \node[fill=white,shape=circle,draw] (Y) at (0,1) {$y$};
    \draw (-2,0) -- (0,-.5) -- (2,0)-- (Y) -- (-1,1) -- (-2,0);
    \draw (Y) -- (1, 1);
    \node[fill=white,shape=circle,draw] at (-1,1) {$x$};
    \node[fill=white,shape=circle,draw] at (1,1) {$z$};
    \node[fill=white,shape=circle,draw] at (-2,0) {$l$};
    \node[fill=white,shape=circle,draw] at (2,0) {$r$};
    \node[fill=white,shape=circle,draw] at (0,-.5) {$c$};
  \end{tikzpicture}
  \hfill\-
  \begin{tikzpicture}
    \draw (-2,0) -- (0,-.5) -- (2,0)-- (1,1) -- (0,1)--(-1,1) -- (-2,0);
    \draw[dashed,rounded corners=3mm] (-1.75,0.25) rectangle (1.75,1.75);
    \node[fill=white,shape=rectangle,draw,inner xsep=3mm,inner ysep=2mm,dashed,rounded corners=3mm] at (0,1) {$\inside{G}$};
    \node[fill=white,shape=circle,draw] at (-1,1) {$l$};
    \node[fill=white,shape=circle,draw] at (1,1) {$r$};

    \node[fill=white,shape=circle,draw] at (-2,0) {$l'$};
    \node[fill=white,shape=circle,draw] at (2,0) {$r'$};
    \node[fill=white,shape=circle,draw] at (0,-.5) {$c$};
  \end{tikzpicture}
  \hfill\-
  \caption{\emph{Left} a graph  that codes ``true''.
    \emph{Right}, a graph that codes negation.
  }
  \label{figure.codingTrueNeg}
\end{figure}

\begin{lemma}\label{lem.codingand}
  There is a single computable procedure which, given any two coding graphs $\la G_i, l_i, r_i, c_i \ra$, $i \in \{1, 2\}$, each with a unique perfect matching, and where $G_i$ codes  $P_i$, returns another coding graph $\la G, l, r, c \ra$ with a unique matching, that codes $\neg P_1 \land P_2$.

  Together with the ability to code negations, it follows that we can construct coding graphs to code $P_1 \land P_2$, $P_1 \lor P_2$, and $P_1 \to P_2$.
\end{lemma}

\begin{proof}
  Given the coding graphs $\la G_i, l_i, r_i, c_i \ra$ coding $P_i$, $i \in \{1, 2\}$, define a coding graph $\la G, l, r, c \ra$ that codes $\neg P_1 \wedge P_2$ by removing the vertices $c_i$, then adding a new vertex $rr$, and including the cycle $l_1$-$l_2$-$r_1$-$r_2$-$rr$-$r$-$c$-$l$-$l_1$ (\Cref{figure.codingand.01}).

By assumption, the component graphs $G_1$ and $G_2$ have unique perfect matchings.
It is straightforward to show that these extend to a unique perfect matching of $G$. 
There are four cases, depending on the truth value coded by the $G_i$.

For example, if $G_1$ and $G_2$ both code true, then $l_1$ matches into $G_1$ and $l_2$ matches into $G_2$.  
Then the matching must include $\{l, c\}$, $\{r, rr\}$, and $\{r_2, r_1\}$ leading to a unique perfect matching where $G$ codes false, matching the desired truth value of $\neg T\land T\equiv F$. 
The proofs of the other three cases are similar. 
\end{proof}

  %
  %
  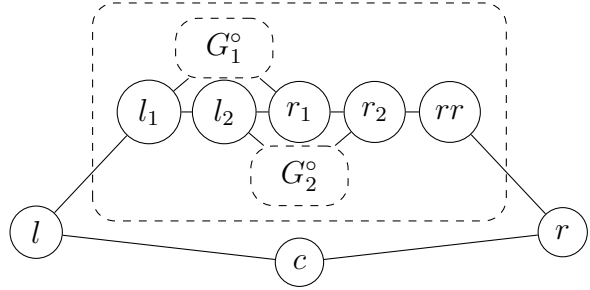
\begin{figure}[htbp]
    \begin{center}
      \begin{tikzpicture}

        \draw[dashed,rounded corners=3mm] (-1.75,-1.45) rectangle (3.75,1.45);  

        \node[fill=white,shape=rectangle,draw,inner xsep=4mm,inner ysep=2mm,dashed,rounded corners=3mm] (G1) at (0,.85) {$\inside{G}_1$}; 
        \node[fill=white,shape=circle,draw] (L1) at (-1,0) {$l_1$};
        \node[fill=white,shape=circle,draw] (R1) at (1,0) {$r_1$};
        \node[fill=white,shape=rectangle,draw,inner xsep=4mm,inner ysep=2mm,dashed,rounded corners=3mm] (G2) at (1,-.85) {$\inside{G}_2$};
        \node[fill=white,shape=circle,draw] (L2) at (0,0) {$l_2$};
        \node[fill=white,shape=circle,draw] (R2) at (2,0) {$r_2$};	
        \node[fill=white,shape=circle,draw] (L) at (-2.5,-1.6) {$l$};
        \node[fill=white,shape=circle,draw] (R) at (4.5, -1.6) {$r$};
        \node[fill=white,shape=circle,draw] (C) at (1,-2) {$c$}; 
        \node[fill=white,shape=circle,draw] (RR) at ( 3,0) {$rr$};

        \draw (G1)--(L1) -- (L2);
        \draw (R1)--(G1);
        \draw (G2)--(L2)--(R1) -- (R2)--(G2);
        \draw (L1) -- (L) -- (C) -- (R) -- (RR)--(R2);
      \end{tikzpicture}
    \caption{Coding $\neg P_1 \land P_2$, where $P_i$ is coded by  $\la G_i, l_i, r_i, c_i\ra$}    \label{figure.codingand.01}
    \end{center}

  \end{figure}

\begin{lemma}\label{lem.codingsets}
Let $\la G_i  \ra$ be a sequence of disjoint coding graphs, each with a unique matching.  Then $\bigsqcup G_i$ can be seen as coding membership in the set $\{i: G_i$ codes true$\}$.
\end{lemma}
\begin{proof}
Because the graphs are disjoint, 
if all of the given graphs satisfy condition~(A) and have perfect matchings, then the union will satisfy condition~(A) and have a perfect matching.  
Furthermore, any perfect matching of the whole graph will be unique, and code the truth of all the component propositions.
\end{proof}

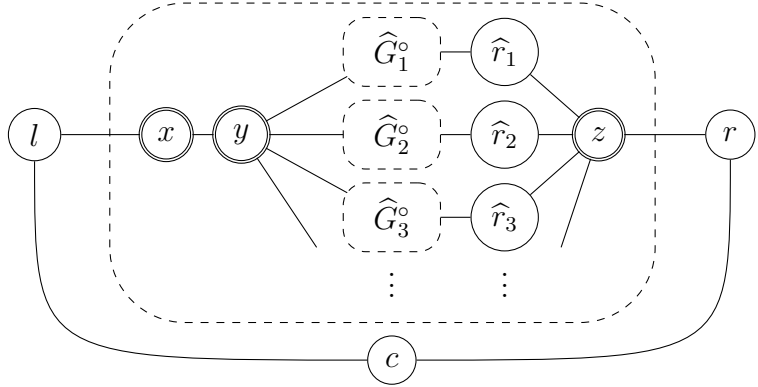
\begin{figure}[htbp]
  \begin{center}
    \begin{tikzpicture}
      \draw[dashed,rounded corners=10mm] (-3.75,-2.5) rectangle (3.5,1.75);
			\node (G1) at (0,1.1) {};
      \node[fill=white,shape=circle,draw] (r1) at (1.5,1.1) {$\hat r_1$};
			\node (G2) at (0,0) {};
      \node[fill=white,shape=circle,draw] (r2) at (1.5,0) {$\hat r_2$};
			\node (G3) at (0,-1.1) {};
      \node[fill=white,shape=circle,draw] (r3) at (1.5,-1.1) {$\hat r_3$};
      \node at (0,-1.9) {$\vdots$};
      \node at (1.5,-1.9) {$\vdots$};
      \node[fill=white,shape=circle,draw] (l) at (-4.75,0) {$l$};
      \node[fill=white,shape=circle,draw] (r) at ( 4.5,0) {$r$};
      \node[fill=white,shape=circle,,draw] (c) at (0,-3) {$c$};
      \node[fill=white,shape=circle,draw,double] (x) at (-3,0) {$x$};
      \node[fill=white,shape=circle,draw,double] (y) at (-2,0) {$y$};
      \node[fill=white,shape=circle,draw,double] (z) at (2.75,0) {$z$};
      \draw (y) -- (G1) -- (r1) -- (z) -- (r2) -- (G2) -- (y) -- (G3) -- (r3) -- (z);
			\draw (l)  .. controls (-4.75,-3) .. (c) .. controls (4.5,-3) .. (r);
      \draw (l) -- (x)--(y);
      \draw (z) -- (r);
			\draw (y) --++ ( 1  ,-1.5);
			\draw (z) --++ (-0.5,-1.5);
      \node[fill=white,shape=rectangle,draw,inner xsep=4mm,inner ysep=2mm,dashed,rounded corners=3mm] at (G1) {$\inside{\hat G}_1$};
      \node[fill=white,shape=rectangle,draw,inner xsep=4mm,inner ysep=2mm,dashed,rounded corners=3mm] at (G2) {$\inside{\hat G_2}$};
      \node[fill=white,shape=rectangle,draw,inner xsep=4mm,inner ysep=2mm,dashed,rounded corners=3mm] at (G3) {$\inside{\hat G_3}$};
    \end{tikzpicture}
  \caption{Coding $(\exists i) P(i)$ in the case where $P(i)$ is true for at most one $i$.}  \label{figure.existsgadget}
  \end{center}

\end{figure}

\begin{lemma}\label{lem.codingexists}
  There is a single computable procedure which, given a uniformly computable sequence of coding graphs $\langle G_n , l_n, r_n, c_n\rangle$, each with a unique matching and each coding a predicate $P(n)$, returns a coding graph with a unique matching that codes the truth of the sentence $(\exists n) P(n)$.  
\end{lemma}
\begin{proof}
  For technical reasons, we must begin by using negation and conjunction to obtain a new sequence of coding graphs $\hat G_i$ which code the predicates $\hat P(i) = P(i) \land \neg P(i-1) \land \dots \land \neg P(1)$. 
  Then at most one of the $\hat G_i$ codes true (and if there is exactly one, it will be the $\hat G_i$ with $i$ least possible).  Note also that $(\exists n) P(n)\equiv (\exists n) \hat P(n)$.

We can now create a new coding graph $\la G, l, r, c \ra$ that will code $(\exists n) \hat P(n)$.  We modify and combine the coding graphs $\hat G_i$ as in \Cref{figure.existsgadget}:  For each $i$, the vertices $\hat l_i$ and $\hat c_i$ are removed.
  The vertices of the graph being constructed consist of the remaining vertices of the coding graphs $\inside{\hat G_i}$ together with new vertices: $x, y, z$.
  for each $i$, the unique vertex in $\inside{\hat G_i}$ that was adjacent to $\hat l_i$ is set adjacent to the single vertex $y$, and all the vertices $\hat r_i$ are set adjacent to the separate vertex $z$.  Finally, we include the path $y$-$x$-$l$-$c$-$r$-$z$.

We must show that there exists a unique matching of this graph, and that it codes $(\exists n)\hat P(n)$.  Since at most one of the predicates $\hat P(i)$ is true, we have two cases. 

First, suppose there exists exactly one $k$ such that $\hat P(k)$ is true.  Then from the unique perfect matching of $\hat G_k$, we get that $y$ must be matched into $\inside{\hat G_k}$, which means $\hat r_k$ must be matched to $z$ and all other $\hat r_i$ must be matched into their respective $\inside{\hat G_i}$.  Furthermore, $\{r, c\}$ and $\{l,x\}$ must be in the matching, giving a unique perfect matching of $G$.

Second, suppose that all $\hat P(i)$ are false.
Then from the unique perfect matchings of each $\hat G_k$, each $\hat r_i$ must be matched  into its respective $\inside{\hat G_i}$.  This means $\{z, r\}$, $\{c, l\}$, and $\{x, y\}$ must be in the matching, again giving a unique perfect matching of $G$.
\end{proof}

  Suppose we have coding graphs $G_0, G_1,G_2, \dots$ such that $G_i$ codes $i \in X$, and that we wish to code $\exists n \, [{\varphi^{X}_{e,n}(e)\halts}]$.
  Unfortunately, our procedure does not have access to the set $X$ itself.  Because $X$ is only \emph{coded}, it can only be computed  from perfect \emph{matchings} of the $G_i$.

\begin{lemma}\label{lem.codingjumps}
  There is a uniformly computable procedure which,    
  given any sequence of graphs $\langle G_i\rangle$ coding the predicates $ i \in X $ and given any $e \in \N$, returns a single graph that codes the predicate $e \in X'$.     
\end{lemma} 
\begin{proof} 
  As usual, we adopt the convention that ${\varphi^{X}_{e,n}(e)\halts}$ if and only if  ${\varphi^{X \upto n}_{e,n}(e)\halts}$.  Furthermore, we will use the characteristic function of $X$ in the place of $X$ and  denote an arbitrary initial segment of $X$ by $\sigma$, a finite string of $0$'s and $1$'s.
  Recall that there is a uniformly computable function $f$ such that for each $e,n \in \N$ and each $\sigma \in 2^{n}$, $\varphi_e^{\sigma}(e) = \varphi_{f(e, \sigma)}(e)$.  Thus the statements  ${\varphi_{f(e,\sigma),n}(e)\halts}$ clearly have uniformly given coding graphs.

  Suppose we wished to perform a computation relative to ${X \upto n}$.
  Without knowing the perfect matchings of the $G_i$, each string $\sigma\in 2^{n}$ is a possible initial segment of $X$.
  To code ${\varphi^X_{e,n}(e)\halts}$ without any knowledge of $X$, we code the following statement, which accounts for every possible initial segment $\sigma$ of $X$.
  \[\bigwedge_{\sigma \in 2^n} \left[
      \left(\bigwedge_{\sigma(j) = 1} (j\in X) \ \land \bigwedge_{\sigma(j) = 0} \neg (j\in X)
      \right) \to {\varphi_{f(e, \sigma), n}(e)\halts}\right]\]

  Note that for each $\sigma\not\prec X$, its corresponding antecedent will be false, so that particular conjunction will be true.
  For the unique $\sigma\prec X$ of length $n$, its corresponding antecedent will be true, and the consequent will determine the truth value of the whole conjunction.  
  To determine if ${\varphi_e^X(e)\halts}$, simply code $\exists n \, [{\varphi_{e,n}^X(e)\halts}]$ as usual.
\end{proof}

Putting everything together, we obtain the following.

\begin{theorem} \label{codeonehypset}
Let $X \subseteq \N$.  For every $e \in \O^X$, there exists an $X$-computable graph $G_e^X$ such that $H_e^X$ is computable in any perfect matching of $G_e^X$. 
\end{theorem}

\begin{proof}
Lemmas~\ref{lem.codingtruenegation} and \ref{lem.codingsets} provide a uniform procedure for coding computable sets by sequences of graphs.  
For successor ordinals, \Cref{lem.codingjumps} shows that given a sequence of graphs coding a set $X$, we can uniformly find a sequence of graphs coding $X'$.  
For limit ordinals, given any computable simultaneous sequences of graphs $\langle G_{i,n} \rangle$ coding the sets $X_n$, we can uniformly produce a sequence of graphs coding the effective join $\bigoplus X_n = \{\la n, i \ra : i \in X_n\}$.  Here we can simply take $\langle G_{\la i, n \ra} \ra$.
By effective transfinite recursion, a consequence of the recursion theorem, there is therefore a uniformly computable sequence such that for every $e \in \O^X$, there is a graph $G_e^X$, which is computable in $X$, such that $H_e^X$ is computable in any perfect matching of $G_e^X$. 
\end{proof}

Implicit in the above construction is the fact that for each $e \in \O^X$, $G_e^X$ satisfies condition~(A).  Because we have given a transfinite recursive construction that preserves condition~(A) at each successor and limit sage, it follows by  transfinite induction that all resulting graphs satisfy condition~(A).

Because condition~(A) is  $\Pi^1_2$, the transfinite induction used above is  $\Pi^1_2$-transfinite induction.  
The system, known as $\Pi^1_2$-$\mathsf{TI}_0$, is slightly weaker than $\Sigma^1_2$-$\mathsf{DC}_0$, which is equivalent to $\Delta^1_2$-$\mathsf{CA}_0$ plus $\Sigma^1_2$-$\mathsf{IND}$ by Theorem~VII.6.9 of \cite{simpson}.
Thankfully, the level of induction can be reduced using the fact that these coding graphs do not have any infinite paths.

\begin{lemma}\label{codinggraphfinitepath} \label{codingsatA}
The following hold over $\Pi^1_1$-$\mathsf{TI}_0$.
\begin{enumerate}
\item For any $e$, if $e \in \O^X$, then the coding graph $G_{e}^X$ has no infinite paths.
\item  For any $e$, if $e \in \O^X$, then the coding graph $G_{e}^X$ satisfies condition~(A).
\end{enumerate}
\end{lemma}

\begin{proof}
The proof of (1) is a straightforward transfinite induction argument.  
Using (1), the transfinite induction used in (2) becomes $\Pi^1_1$. 
\end{proof}

This gives us the following strengthened result.
\begin{corollary} \label{reversetoATR}
 $\FPST + \Pi^1_1$-$\mathsf{TI}_0$ implies $\ATR$ over $\RCA$.
\end{corollary}
\begin{proof}  This is a direct consequence of \Cref{codeonehypset}, together with the above lemma.  As $\ST$ already implies $\ACA$ over $\RCA$, everything can be done over $\ACA$.  Fix a set $X$ and an $e \in \mathbb{N}$.  Assume $e \in \O^X$.  Apply \Cref{codeonehypset} to obtain the coding graph $G_e^X$.  By \Cref{codingsatA}, together with $\Pi^1_1$-$\mathsf{TI}_0$, we have that $G_e^X$ satisfies condition~(A).  Then $\FPST$ provides a perfect matching $M$ of $G_e^X$, from which we can compute $H_e^X$, again by \Cref{codeonehypset}.
\end{proof}

Although the results above do not give a complete reversal from $\ST$ to $\ATR$, they will enable us to separate $\FPST$ from $\SAC$ and $\SDC$.
By Theorem~VIII.5.12 of \cite{simpson}, $\Pi^1_1$-$\mathsf{TI}_0$ is equivalent to $\SDC$ over $\ACA$.  

\begin{corollary}
  \label{corollary.steffens-not-consequence-of-Sigma11AC}
  Neither $\ST$ nor $\FPST$ are implied by $\SAC$, $\SDC$, or $\CPM$.  
  In particular, $\ST$ is strictly stronger than $\SAC$.
\end{corollary} 
\begin{proof}
We first show that $\HYP$, the $\omega$-model consisting of the hyperarithmetical sets, is not an $\omega$-model of  $\FPST$ (or $\ST$). 
By \Cref{reversetoATR}, $\FPST$ plus $\SDC$ implies $\ATR$ (and similarly for $\ST$), so every model of $\FPST$ plus $\SDC$ is a model of $\ATR$.
But by Proposition~V.2.6 in \cite{simpson}, $\HYP$ is not a model of $\ATR$. 
 
On the other hand, by Corollary~VIII.4.17 of \cite{simpson}, $\HYP$ is an $\omega$-model of $\SAC$ and an $\omega$-model of $\SDC$.  Hence, it is an $\omega$-model of $\SAC + \SDC$.

It follows that $\FPST$ cannot be a consequence of $\SAC$ (and similarly for $\ST$).
By \Cref{disjointunion}, $\CPM$ is equivalent to $\SAC$, so neither $\ST$ nor $\FPST$ follow from $\CPM$.
\end{proof}

\begin{question}
Does $\FPST$ imply $\ATR$ over $\RCA$?  This would be true if there is a reversal from $\FPST$ to $\SDC$ over $\SAC$.
Does $\ATR$ prove either $\FPST$ or $\ST$?
\end{question}

Our results show that $\FPST$ is quite close in strength to $\ATR$.

\begin{conjecture}
$\FPST$ is equivalent to $\ATR$ over $\RCA$.
\end{conjecture}

\subsection{Weihrauch reductions} \label{sect.weihrauch}

Although it is not the main focus of our paper, the reversal in previous section is perhaps best stated in a special case of the language of Weihrauch reducibility.  For a detailed introduction and background, see \cite{Weihrauch-survey}.

In the language of Weihrauch reducibility, mathematical problems are represented by multivalued partial functions, which are written as $f:{\subseteq A} \rightrightarrows B$, where $A, B$ contain of mathematical objects, represented as subsets of $\N$ via the standard codings. For each instance $X\in \dom(f)$, $f(X)$ is the set of all solutions to this instance of the problem.  A \emph{realizer} for the problem $f:{\subseteq A} \rightrightarrows B$ is a function  $F:{\subseteq A} \rightarrow B$ that assigns exactly one solution to each instance of the problem (for each $X\in A$, $F(X)$ is an $F(X) \in f(X)$).  

In this language, $\ST$ corresponds to the partial multivalued function $f:{\subseteq\mathsf{Graphs}} \rightrightarrows \mathsf{Matchings}$ whose domain is the set of graphs satisfying condition~(A) and so that for each such graph $X$, $f(X)$ is the set of all perfect matchings of $X$.   A realizer $F$ of this $f$ is any function that assigns a \emph{specific} perfect matching to each $X \in \dom(f)$.  To compare $\ST$ with $\ATR$, it suffices to use a special case of the Weihrauch Reducibility defined in \cite{Weihrauch-survey,kihara}.

\begin{define}
Let $f, g$ be multivalued functions representing problems in countable mathematics.  Then $f$ is Weihrauch reducible to $g$ if there are computable functions $K, H:{\subseteq \N^{\N}} \rightarrow \N^{\N}$ such that for any $G$, if $G$ is a realizer for $g$, then the function defined by $F(p) = K\langle p, G \circ H(p) \rangle$ is a realizer for $f$.
\end{define}

In \cite{kihara}, Kihara, Marcone and Pauly give a definition for the Weihrauch principle $\mathsf{ATR}$, and show that it is strongly Weihrauch equivalent to the principle ``Unique Closed Choice on Baire Space,'' written $\mathsf{UC}_{\N^{\N}}$.  And while there is no single Weihrauch principle which can be considered the Weihrauch analogue of $\ATR$, the principle $\mathsf{UC}_{\N^{\N}}$,  and therefore $\mathsf{ATR}$, is Weihrauch equivalent to many principles which are equivalent to $\ATR$ in reverse mathematics.

In particular, we will use a result of Goh \cite{goh}, that $\mathsf{ATR}$ is Weihrauch equivalent to the problem whose instances are pairs $(\mathcal{L}, A)$ such that $\mathcal{L} = (L, 0_L, S, p)$ is a labeled well-ordering and $A \subseteq \N$, and such that for each instance $(\mathcal{L}, A)$, the unique solution is the jump hierarchy $\la X_a \ra_{a \in L}$ starting with $A$.

\begin{corollary} \label{weihrauchATR}
$\mathsf{ATR}$ is Weihrauch reducible to $ \FPST$.
\end{corollary}
\begin{proof}
Consider any pair $(\mathcal{L}, A)$, where $\mathcal{L} = (L, 0_L, S, p)$ is a labeled well-ordering and $A \subseteq \N$.  Perform a similar construction used to prove \Cref{codeonehypset}, using $(\mathcal{L} \oplus A)$-transfinite recursion on $L$ to obtain a graph $G$ such that any perfect matching of $G$ codes the jump hierarchy on $L$ which starts with $A$.  We can use transfinite induction to show that $G$ has no infinite paths, and that $G$ satisfies condition~(A).  By $\FPST$, $G$ has a perfect matching $M$, which codes the jump hierarchy on $L$ starting with $A$, as desired.
\end{proof}

\appendix

\bibliographystyle{plain}
\bibliography{matchings-rm}

\end{document}